\documentclass[11pt]{amsart}
\usepackage{amssymb,mathrsfs,graphicx,enumerate,mathabx}
\usepackage{amsmath,amsfonts,amssymb,amscd,amsthm,bbm}
\usepackage[retainorgcmds]{IEEEtrantools}
\usepackage{colortbl}
\usepackage{mathtools}
\mathtoolsset{showonlyrefs}
\allowdisplaybreaks
\title[CS model on $\mathrm{SO}(3)$]{Emergent behaviors of rotation matrix flocks}

\author[Fetecau]{Razvan C. Fetecau}
\address[Razvan C. Fetecau]{\newline Department of Mathematics\newline Simon Fraser University, 8888 University, Burnaby, BC V5A 1S6, Canada.}
\email{razvan\_fetecau@sfu.ca}

\author[Ha]{Seung-Yeal Ha}
\address[Seung-Yeal Ha]{\newline Department of Mathematical Sciences\newline Seoul National University, Seoul 08826, and \newline
Korea Institute for Advanced Study, Hoegiro 85, Seoul 02455, Republic of Korea}
\email{syha@snu.ac.kr}

\author[Park]{Hansol Park}
\address[Hansol Park]{\newline Department of Mathematical Sciences\newline Seoul National University, Seoul 08826, Republic of Korea}
\email{hansol960612@snu.ac.kr}

\newtheorem{theorem}{Theorem}[section]
\newtheorem{lemma}{Lemma}[section]

\newtheorem{proposition}{Proposition}[section]
\newtheorem{remark}{Remark}[section]

\newtheorem{definition}{Definition}[section]

\newcommand{\bbr}{\mathbb R}

\newcommand{\bbs}{\mathbb S}

\newcommand{\calR}{\mathcal{R}}
\newcommand{\calQ}{\mathcal{Q}}
\newcommand{\calC}{\mathcal{C}}
\newcommand{\dRQ}{d_{RQ}}

\newcommand{\bx}{\mathbf{x}}
\newcommand{\by}{\mathbf{y}}

\newcommand{\bn}{\mathbf{n}}
\newcommand{\bu}{\mathbf{u}}
\newcommand{\bv}{\mathbf{v}}
\newcommand{\ba}{\mathbf{a}}

\newcommand{\hba}{\widehat{\mathbf{a}}}
\newcommand{\hbu}{\widehat{\mathbf{u}}}
\newcommand{\hbv}{\widehat{\mathbf{v}}}
\newcommand{\hbn}{\widehat{\mathbf{n}}}
\newcommand{\hbx}{\widehat{\mathbf{x}}}
\newcommand{\hby}{\widehat{\mathbf{y}}}
\newcommand{\Tr}{\mathrm{tr}}
\newcommand{\Rz}{R^0}

\newcommand{\az}{\mathbf{a_0}}

\newcommand{\veeA}{\mathbf{a}}
\newcommand{\hyp}{\mathcal{H}}

\makeatletter
\@namedef{subjclassname@2020}{\textup{2020} Mathematics Subject Classification}
\makeatother

\begin{document}

\date{\today}

\subjclass[2020]{70G60, 82C10} \keywords{Cucker-Smale model, special orthogonal group, emergence, flocking, velocity alignment, rotation group.}

\thanks{\textbf{Acknowledgment.} R.F. acknowledges support from NSERC Discovery Grant PIN-341834 during this research.
The work of S.-Y. Ha was supported by National Research Foundation of Korea (NRF-2020R1A2C3A01003881).}

\begin{abstract}
We derive an explicit form for the Cucker-Smale (CS) model on the special orthogonal group $\mathrm{SO}(3)$ by identifying closed form expressions for geometric quantities such as covariant derivative and parallel transport in exponential coordinates. We study the emergent dynamics of the model by using a Lyapunov functional approach and La Salle's invariance principle. Specifically, we show that velocity alignment emerges from some admissible class of initial data, under suitable assumptions on the communication weight function. We characterize the $\omega$-limit set of the dynamical system and identify a dichotomy in the asymptotic behavior of solutions. Several numerical examples are provided to support the analytical results. 
\end{abstract}

\maketitle \centerline{\date}


\section{Introduction}
Collective behaviors of complex biological systems are ubiquitous in nature, e.g., aggregation of bacteria \cite{K-S, T-B},  synchronous flashing of fireflies \cite{B-B}, synchronization of rhythms in pacemaker cells \cite{Pe}, flocking of migratory birds \cite{B-C, Rey} and swarming of fish \cite{Ao, D-M, T-T}. We refer to \cite{A-B-F, O-M, P-E-G, V-Z} for a brief survey of collective dynamics and engineering applications in the decentralized control of multi-agent systems. Among them, our main interest lies in flocking behaviors in which particles organize themselves from a disordered state to an ordered motion using simple rules based on environmental information. Despite its ubiquitous presence, systematic mathematical studies have begun only several decades ago by Vicsek et al. \cite{V}, a group of statistical physicists who were mainly motivated by Reynolds's simulations \cite{Rey} of bird flocking in computer graphics. Since then, several mechanical and phenomenological models have been used to study flocking behavior. 

In this paper, we are interested in the Cucker-Smale (CS) model \cite{C-S} for flocking of a matrix ensemble on the special orthogonal matrix group $\mathrm{SO}(3)$, given by:
\begin{equation} \label{A-0-0}
 \mathrm{SO}(3):= \{R\in\bbr^{3\times 3}: R^\top R=I\text{ and }~ \mathrm{det} ~R=1\}, 
\end{equation}
where $\bbr^{3 \times 3}$ is the matrix group consisting of all $3\times 3$ real matrices and $I$ denotes the $3 \times 3$ identity matrix.  Note that $\mathrm{SO}(3)$, also referred to as the rotation group, is a Lie group, having a group and a manifold structure at the same time.  Elements in $ \mathrm{SO}(3)$ are called rotation matrices; they are characterized by an axis and an angle of rotation. The rotation group is the configuration space of a rigid body in $\bbr^3$ that undergoes rotations only (no translations).  There has been recent growing interest in studying collective behavior on $ \mathrm{SO}(3)$  \cite{D-F-M, D-D-F-M, H-K-S}, motivated by body attitude coordination of interacting agents in a variety of applications. For example, engineering applications include synchronization of satellite attitudes \cite{B-P-G, K-G, L-B} and camera pose averaging \cite{TronVidalTerzis2008}. In biology, it is very common for animals such as birds and fish to coordinate their body attitudes -- see  \cite{D-F-M} for some images of collective motions in birds and dolphins.

Before we present the CS model on $\mathrm{SO}(3)$, we introduce the CS model on the Euclidean space $\bbr^d$.
Consider an ensemble of $N$ identical particles with unit mass in $\bbr^n$, and let $\bx_i$ and $\bv_i$ be the position and velocity of the $i$-th particle, respectively. The CS model reads:
\begin{equation}
\begin{cases} \label{A-0}
\displaystyle {\dot \bx}_i  = \bv_i, \quad t > 0, \quad i  = 1, \cdots, N, \\
\displaystyle {\dot \bv}_i  = \frac{\kappa}{N} \sum_{k =1}^{N} \phi_{ik}(\bv_k - \bv_i),
\end{cases}
\end{equation}
where $\kappa$ is a nonnegative coupling strength and $\phi_{ik} = \phi(|\bx_k - \bx_i|)$ are communication weights given in terms of a bounded and Lipschitz continuous communication function $\phi$. In this setting, the global well-posedness of \eqref{A-0} follows directly from the Cauchy-Lipschitz theory together with uniform boundedness of the kinetic energy.  The CS model \eqref{A-0} exhibits interesting flocking dynamics depending on the initial configuration, system parameters and communication functions (see \cite{C-S, H-L, H-T}), and it has been studied extensively in applied mathematics and control communities in the last decade. The model has been investigated in diverse directions and areas of applications, e.g., network topology \cite{C-D2, C-D3, D-H-J-K, D-Q, L-H, L-X}, global and local flocking dynamics \cite{C-S, H-L, H-T, M-T1}, interaction with nonlinear damping \cite{F-H-J},  time-delayed interactions \cite{E-H-S}, kinetic and hydrodynamic CS equations \cite{F-H-T, L-S, P-S, R-S, S-T1, S-T2}, coupling with thermodynamics \cite{H-K-L, H-K-M-R-Z, H-K-Rug, H-R}, complete predictability \cite{H-K-P-Z}, extension to manifolds \cite{A-H-S-0, A-H-S, A-H-P-S, H-K-S} and relativistic CS model \cite{H-K-Rug-0}. We refer to recent survey articles \cite{A-B-F, C-H-L, M-T} for overview and further references.
\vspace{0.2cm}

In the present paper we address the``{\it Cucker-Smale flocking realizability problem"} on $\mathrm{SO}(3)$:
\begin{itemize}
\item
(Q1):~Is there a CS type flocking model on $\mathrm{SO}(3)$? 
\vspace{0.1cm}

\item
(Q2): ~If so, under what conditions on system parameters and initial data, can we guarantee emergent dynamics?
\end{itemize}

\vspace{0.2cm}

To answer question (Q1), we need to replace the R.H.S. of $\eqref{A-0}_2$ by a suitable internal force so that  particles' positions stay on $\mathrm{SO}(3)$ through the time evolution. This is not obvious at first glance, as we need to compare velocities belonging to different tangent spaces. In fact, question (Q1) was already answered affirmatively in an {\it abstract} setting in which the underlying manifold $({\mathcal M}, g)$ is a connected, complete and smooth Riemannian manifold with a metric $g$ (see \cite{H-K-S}), whereas question (Q2) was answered only for the sphere and the hyperboloid (see \cite{A-H-S, A-H-P-S}).


Specifically,  let (${\mathcal M}, g)$  be a connected, complete and smooth Riemannian manifold with metric $g$. Denote by $\frac{Dv}{dt}$ the covariant derivative of a tangent vector $v$. For $\bx_k,\bx_i \in \mathcal{M}$, let $P_{ki}$ be the parallel transport along the length minimizing geodesic from $\bx_k$ to $\bx_i$. In this abstract Riemannian setting, the CS model on $(M, g)$ takes the following form:
\begin{align}\label{A-1}
\begin{cases}
\displaystyle \dot{\bx}_i =\bv_i,\quad t>0,\quad i = 1, \cdots, N, \\
\displaystyle  \frac{D\bv_i}{dt}  =\frac{\kappa}{N}\sum_{k=1}^N\phi_{ik} (P_{ki}\bv_k- \bv_i),
\end{cases}
\end{align}
where $\kappa$ is a nonnegative coupling strength and $\phi_{ik}=\phi(d(\bx_i, \bx_k))$ are communication weights. Here, $d(\cdot, \cdot)$ denotes the geodesic distance on $\mathcal{M}$, and the communication weight function $\phi$ is assumed to be a nonnegative, bounded, and smooth function on $\bbr_{\geq0}$. Note that the parallel transport in model \eqref{A-1} is well-defined only if the length minimizing geodesic between $\bx_k$ and $\bx_i$ is unique, i.e., when the two points are not in the cut locus of each other. To alleviate this issue, if two points are in the cut locus of each other
then we set $\phi$ to be zero for such pairs of points\footnote{In the paper we will use the term {\em pair of cut points} to refer to two points on a manifold that are in the cut locus of each other.}, so that system \eqref{A-1} is well-defined and the local well-posedness of \eqref{A-1}  is guaranteed. For certain special manifolds, one can explicitly calculate $\frac{D\bv_i}{dt}$ and $P_{ki}\bv_k$, and write the abstract CS model \eqref{A-1} in an explicit form. Examples of such manifolds are the Poincar\'e Half plane \cite{H-K-S}, unit sphere \cite{A-H-S, H-K-S} and hyperboloid \cite{A-H-P-S}. For related first-order aggregation models on Riemannian manifolds, we also refer to \cite{A-M-P, F-H-P, F-P-P, F-Z, Ma,S-S, S-B-S}. 

We adopt the following definition of {\it velocity alignment}.  
\begin{definition} \label{D1.1}
\emph{\cite{H-K-S}}
Let $\{(\bx_i, \bv_i)\}_{i=1}^N$ be a global smooth solution to \eqref{A-1}. We say that the configuration exhibits {\it asymptotic velocity alignment} if the following relation holds: 
\[ \lim_{t \to \infty} \max_{1 \leq i, k \leq N}  \| P_{ki}\bv_k(t) - \bv_i(t) \|_{\bx_i} = 0, \]
where $\| P_{ki}\bv_k - \bv_i \|^2_{\bx_i}  = g_{\bx_i}( P_{ki}\bv_k - \bv_i, P_{ki}\bv_k - \bv_i)$.
\end{definition}

Next, we discuss briefly our main results. To simulate the particle system on a specific manifold, the abstract model \eqref{A-1} cannot be used directly, unless we find an explicit form for equation $\eqref{A-1}_2$.  Our first main result is concerned with the explicit derivation of the CS model on $\mathrm{SO}(3)$ by finding closed-form representations for $ \frac{Dv_i}{dt} $ and $P_{ki} \bv_k$. For the model on $\mathrm{SO}(3)$ we will use $R_i(t)$ and $V_i(t) =\frac{dR_i(t)}{dt}$ instead of $\bx_i(t)$ and $\bv_i(t)$, to denote the position and velocity of the $i$-th particle at $t > 0$. Since $V_i(t)$ lies in the tangent space of $\mathrm{SO}(3)$ at $R_i(t)$, there exists a unique skew-symmetric $3 \times 3$ matrix $A_i(t)$ (see Section \ref{sec:2}) such that
\[
V_i(t)=R_i(t)A_i(t).
\]

To close the dynamical system we need to find the dynamics for $V_i$ or $A_i$ from $\eqref{A-1}_2$. In order to present this explicit form, we need some preparations. First note that $\mathfrak{so}(3)$, the Lie algebra of $\mathrm{SO}(3)$ consisting of all $3\times 3$ skew-symmetric matrices, is isomorphic to $\bbr^3$ via the hat map:
\begin{align}\label{eqn:hat}
\bx = (x_1, x_2, x_3) \in \bbr^3 \quad \Longleftrightarrow \quad 
\widehat{\bx}=\begin{bmatrix}
0&-x_3&x_2\\
x_3&0&-x_1\\
-x_2&x_1&0
\end{bmatrix} \in \mathfrak{so}(3).
\end{align}
For given $i, k \in \{1, \cdots, N\}$, we set 
\begin{align*}
\begin{aligned}
\widehat{\mathbf{u}}_{ki}  &:=\log(R_k^\top R_i)=\frac{\theta_{ki}}{2\sin\theta_{ki}}\big(R_k^\top R_i-R_i^\top R_k\big),  \\[5pt]
\theta_{ki} &:=\|\mathbf{u}_{ki}\|=\arccos\left(\frac{\mathrm{tr}(R_k^\top R_i)-1}{2}\right),\quad \mathbf{n}_{ki} :=\frac{\mathbf{u}_{ki}}{\theta_{ki}},
\end{aligned}
\end{align*}
where $\| \cdot \|$ denotes the Euclidean norm in $\bbr^3$.
The precise meanings of these variables will be given in Section \ref{sec:2}. Equipped with these notations, the CS model on $\mathrm{SO}(3)$ is given by:
\begin{align}\label{A-2}
\hspace{1cm}
\begin{cases}
\displaystyle\frac{dR_i}{dt} =R_iA_i,\quad t>0,\quad i = 1, \cdots, N,\\[5pt]
\displaystyle\frac{d\veeA_i}{dt} =\frac{\kappa}{N}\sum_{k=1}^N\phi_{ik}\left[ \left(1-\cos\frac{\theta_{ki}}{2}\right)(\mathbf{n}_{ki}\cdot \veeA_k)\mathbf{n}_{ki}+\sin\frac{\theta_{ki}}{2}\veeA_k\times \mathbf{n}_{ki}+\cos\frac{\theta_{ki}}{2}\veeA_k-\veeA_i\right ],
\end{cases}
\end{align}
where $\veeA_i \in \bbr^3$ satisfies $\widehat{\veeA}_i = A_i$.

As noted above, parallel transport between a pair of cut locus points is not well-defined, and in such cases we require that the communication weighs between the points is zero. On $\mathrm{SO}(3)$, the injectivity radius is $\pi$ and two rotation matrices $R$ and $Q$ are in the cut locus of each other if and only if $d(R,Q)=\pi$. For system \eqref{A-2} to be well-defined and also locally well-posed, one of the following conditions will be used later.
\begin{itemize}
\item 
$({\hyp}_A)$:  $\phi(R,Q)$ is a continuous function of $d(R,Q)$ that vanishes at pairs of cut locus points:
\[  \phi(R,Q)=\tilde{\phi}(d(R,Q)), \qquad \mbox{and} \qquad  \tilde{\phi}(\pi)=0. \]

\item
$({\hyp}_B)$: for any $1\leq i, k\leq N$ and $t\geq0$, the distance between $R_i(t)$ and $R_k(t)$ is strictly less than $\pi$. 
\end{itemize}

Our second main results deal with the emergent dynamics of \eqref{A-2}. For this, we introduce a Lyapunov functional $\mathcal{E}$ given by the total kinetic energy:  
\[
\mathcal{E}(t)=\sum_{i=1}^N {\| V_i(t)\|}_{R_i(t)}^2.
\]

Then, by straightforward calculations, one can derive a dissipation estimate (see Section \ref{sec:4.3}) :
\begin{equation} \label{A-3}
\frac{d\mathcal{E}}{dt} =-\frac{\kappa}{N}\sum_{i, k=1}^N\phi(R_i, R_k)\|P_{ki} V_k- V_i\|_{x_i}^2\leq0.
\end{equation}
We apply LaSalle's invariance principle \cite{La} and use \eqref{A-3} to obtain
\[
\lim_{t\to\infty}\frac{d\mathcal{E}}{dt} =0, \quad \mbox{or equivalently} \quad \lim_{t\to\infty}\phi(R_i, R_k)\|P_{ki} V_k- V_i\|_{R_i}^2=0\quad\forall~ i, k \in \{1, \dots ,N\}.
\]
Under the assumptions $({\hyp}_A)$ or $({\hyp}_B)$, one can then derive velocity alignment  (Theorem \ref{T4.1}).  

Finally, we characterize the $\omega$-limit set of model \eqref{A-2}, and show that a certain dichotomy holds in the asymptotic behavior of the solutions (Theorem \ref{T4.2}). This result is confirmed by numerical simulations.

The paper is organized as follows. In Section \ref{sec:2}, we review briefly the geometry of the rotation group, e.g., the exponential map, the explicit form of the metric tensor, geodesic equations, as needed for later sections. In Section \ref{sec:3}, we provide equations for geodesics and provide a closed form representation for parallel transport and its geometric interpretation. In Section \ref{sec:4}, we explicitly derive the CS model on $\mathrm{SO}(3)$ from the abstract model \eqref{A-1} and study its emergent dynamics. In Section \ref{sec:5}, we provide several numerical simulations and compare them with the analytical results in previous sections. Finally Section \ref{sec:6} is devoted to a brief summary of our main results and some remaining issues for future work. In Appendix we provide proofs for several lemmas and details to various calculations used in the paper.
\medskip

\noindent {\bf Gallery of Notation}: For notational simplicity, we use the following handy notation:
\[ \max_{i}  := \max_{1 \leq i \leq N}, \qquad  \max_{i, j}  := \max_{1 \leq i, j \leq N}, \qquad \sum_{i} :=  \sum_{i =1}^{N}, \qquad \sum_{i, j} :=  \sum_{i = 1}^{N} \sum_{j= 1}^{N}, \]
and if there is no confusion, we also use Einstein's notation on repeated indices from time to time. Let $\bbr^{3 \times 3}$ be the collection of real $3 \times 3$ matrices and for $A_1, A_2  \in \bbr^{3 \times 3}$, we introduce the Frobenius inner product and its associated norm:
\[
\langle A_1, A_2 \rangle_{\mathrm{F}} := \mbox{Tr}(A_1^\top  A_2), \quad \|A_1 \|_{\mathrm{F}} := \sqrt{\langle A_1, A_1 \rangle_\mathrm{F}}.
\]


\section{Geometry of the special orthogonal group $\mathrm{SO}(3)$}  \label{sec:2}
\setcounter{equation}{0}
In this section, we review minimal background on the geometry of the special orthogonal group $\mathrm{SO}(3)$ which will be necessary to follow discussions in later sections. In particular, we present $\mathrm{SO}(3)$ as a Riemannian manifold and derive explicit calculations using the parametrization by exponential coordinates.
\subsection{General considerations} \label{sec:2.1}  The Lie algebra $ \mathfrak{so}(3)$ of the Lie group $\mathrm{SO}(3)$ in \eqref{A-0-0}, representing the tangent space of $\mathrm{SO}(3)$ at the identity $I$, is given by:
\[ \mathfrak{so}(3) := \{A \in\bbr^{3\times 3}:~A^\top  + A = O \}, \]
where $O$ is the zero matrix. In general, the tangent space at a generic point $R \in \mathrm{SO}(3)$ is given by
\begin{equation}
\label{eqn:tspace}
T_R \mathrm{SO}(3) :=\{RA: A\in\mathfrak{so}(3)\}.
\end{equation}

The Riemannian metric $g$ on $T_R \mathrm{SO}(3)$ is defined as:
\begin{equation}
\label{eqn:Rmetric}
g(RA_1, RA_2) :=\frac{1}{2}\langle RA_1,RA_2\rangle_{\mathrm{F}}=\frac{1}{2}\langle A_1, A_2\rangle_{\mathrm{F}},
\end{equation}
for any $RA_1, RA_2\in T_R\mathrm{SO}(3)$. Note that the norm induced by the metric satisfies 
\begin{equation}
\label{eqn:length}
\| RA \|_{T_R \mathrm{SO}(3)} = \frac{1}{\sqrt{2}} {\| RA \|}_\mathrm{F} = \frac{1}{\sqrt{2}} {\| A \|}_\mathrm{F}.
\end{equation}

Throughout the paper, we use $\|\cdot\|$ for the usual vector norm in $\bbr^3$, $\|\cdot\|_\mathrm{F}$ for the Frobenius norm, and $\|\cdot\|_R$ for the Riemannian norm at $R\in\mathrm{SO}(3)$. We also denote the Frobenius and the Riemannian inner products by ${\langle \cdot, \cdot\rangle}_\mathrm{F}$ and ${\langle \cdot, \cdot\rangle}_R$, respectively. In summary, we have the following relations between the Frobenius and Riemannian norms/inner-products as follows:
\[
{\langle V_1, V_2\rangle}_R=\frac{1}{2} {\langle V_1, V_2\rangle}_\mathrm{F},\qquad {\|V\|}_R=\frac{1}{\sqrt{2}}{\|V\|}_\mathrm{F},
\]
where $V$, $V_1$, and $V_2$ are tangent vectors in $T_{R}\mathrm{SO}(3)$.

\subsubsection{Isomorphism between $\mathfrak{so}(3)$ and $\bbr^3$}  \label{sec:2.1.1} To any vector $\bx=(x_1, x_2, x_3)\in\bbr^3$ we associate its corresponding skew-symmetric matrix $\hbx \in \mathfrak{so}(3)$ given by \eqref{eqn:hat}.
By introducing the canonical basis of $\mathfrak{so}(3)$:
\begin{equation} \label{B-4-0-0}
E_1=\begin{bmatrix}
0&0&0\\
0&0&-1\\
0&1&0
\end{bmatrix},\quad E_2=\begin{bmatrix}
0&0&1\\
0&0&0\\
-1&0&0
\end{bmatrix},\quad E_3=\begin{bmatrix}
0&-1&0\\
1&0&0\\
0&0&0
\end{bmatrix},
\end{equation}
the hat operation can be expressed as:
\begin{equation} \label{B-4-0}
\widehat{}:\bbr^3\to\mathfrak{so}(3),\qquad \mathbf{x}\mapsto \widehat{\bx} = x_\alpha E_\alpha,
\end{equation}
where we used the Einstein convention for summation over repeated indices. Note that the components of $E_\alpha$, $\alpha =1,2,3$ can be also expressed as 
\[
[E_\alpha]_{\beta\gamma}=-\epsilon_{\alpha\beta\gamma},
\]
where $\epsilon$ denotes the Levi-Civita symbol. 

The inverse of the hat map, denoted by $\, \widecheck{}\,$, is given by:
\[
\widecheck{}: \mathfrak{so}(3)\to\bbr^3,\qquad  x_\alpha E_\alpha \mapsto \bx = (x_1, x_2, x_3).
\]
The hat operator (and its inverse) establish an isomorphism between $\bbr^3$ and $\mathfrak{so}(3)$. In particular, $\| \bx \| =  {\| \hbx \|}_I$, for all $\bx \in \bbr^3$. 

\subsubsection{Angle-axis representation} \label{sec:2.1.2} As mentioned in the Introduction, an element in $ \mathrm{SO}(3)$ is called a rotation matrix. Hence, any rotation $R \in \mathrm{SO}(3)$ can be identified with a pair $(\theta,\bv) \in [0,\pi] \times \bbs^2$, where $\bbs^2$ denotes the unit sphere in $\mathbb{R}^3$. The unit vector $\bv$ indicates the axis of rotation and $\theta$ represents the angle of rotation (by the right-hand rule) about the axis $\bv$. The pair $(\theta,\bv)$ is also referred to as the angle-axis representation of the rotation $R$. The expression of $R$ in terms of $(\theta,\bv)$ is given by the exponential map via Rodrigues's formula:
\begin{align}\label{eqn:tvtoR}
R=\exp(\theta \widehat{\mathbf{v}})=I+\sin\theta \widehat{\mathbf{v}}+(1-\cos\theta)\widehat{\mathbf{v}}^2.
\end{align}
The inverse of Rodrigues's formula, $\theta \hbv = \log(R)$, is also given by
\begin{equation}
\label{eqn:Rtotv}
\theta=\arccos\left(\frac{\mathrm{tr}R-1}{2}\right),\quad \widehat{\mathbf{v}}=\frac{1}{2\sin\theta}(R-R^\top ).
\end{equation}

\subsubsection{Geodesic distance, exponential and logarithm maps} \label{sec:2.1.3}  Given two rotation matrices $R$, $Q \in \mathrm{SO}(3)$, the shortest path between $R$ and $Q$ is the geodesic curve $\calR :[0,1] \to \mathrm{SO}(3)$ given by
\begin{equation*}
\label{eqn:Rgeod}
\calR(t) = R \exp(t \log(R^\top  Q)).
\end{equation*}
Note that 
\[ \calR'(t) = \calR(t) \log(R^\top  Q) \in T_{\calR(t)} \mathrm{SO}(3), \]
and  the Riemannian distance between $R$ and $Q$ on on $\mathrm{SO}(3)$ is
\begin{equation*}
\label{eqn:Rdist}
d(R,Q) = \dRQ,
\end{equation*}
where $\dRQ$ satisfies
\[ \dRQ \widehat{\bv}_{_{RQ}}= \log(R^\top  Q). \]
Then by \eqref{eqn:Rtotv}, we also have
\begin{equation}
\label{eqn:geod-dist}
d(R,Q)  =  \operatorname{arccos}\left(\frac{\operatorname{tr}(R^\top  Q) -1}{2} \right),
\end{equation}
and in particular, $d(I,R) = \operatorname{arccos}\left(\frac{\operatorname{tr}R -1}{2} \right)$. 

On the other hand, the Riemannian exponential and logarithm maps at $R$ are given as follows:
\begin{equation} \label{eqn:exp-map}
\begin{cases}
\displaystyle \exp_{R}: T_{R} \mathrm{SO}(3) \to \mathrm{SO}(3), \qquad \exp_{R}(R A) = R \exp (A), \\[5pt]
\displaystyle \log_{R}: \mathrm{SO}(3) \to T_{R} \mathrm{SO}(3), \qquad \log_{R}(Q) = R \log(R^\top  Q).
\end{cases}
\end{equation}
Rodrigues's  formulas \eqref{eqn:tvtoR} and \eqref{eqn:Rtotv} can be expressed using the exponential map at the identity $I$, as
\[
R = \exp_I (\theta \widehat \bv ), \qquad \theta \widehat \bv = \log_I R.
\]



\subsection{Exponential coordinates} \label{sec:2.2}
For fixed $R_0\in \mathrm{SO}(3)$, we introduce exponential coordinates (also commonly referred to as Riemannian normal coordinates \cite{Petersen2006}) on $\mathrm{SO}(3)$ centerd at $R_0$. With this aim we consider the map $R : \bbr^3 \to \mathrm{SO}(3)$ (see \eqref{eqn:exp-map}) given by
\begin{align}\label{B-4}
\bx = (x_1, x_2, x_3) \longrightarrow  R(\bx) = R_0\exp\left(x_ \alpha E_\alpha \right),
\end{align}
which is an injective map on $\| \bx \|  <\pi$. The chart covers all but the cut locus of $R_0$. Throughout the paper, we will use Greek letters ($\alpha,\beta, \gamma$, etc) for coordinates on $\mathrm{SO}(3)$. Consequently, indices in Greek letters take values in $\{1,2,3\}$. We make this distinction to avoid possible confusion with indices for particles in the Cucker-Smale model, for which we use Roman letters ($i,j,k$, etc); such indices take values in $\{1,2,\dots,N\}$. First, we compute the basis $\{ \partial_{1}R(\mathbf{x}), \partial_{2}R(\mathbf{x}), \partial_{3}R(\mathbf{x})\} $ of the tangent space $T_{R(\bx)} \mathrm{SO}(3)$ at $R(\bx)$ induced by the coordinate chart \eqref{B-4}. By differentiating \eqref{B-4}, we have
\begin{equation}
\label{eqn:pR-alpha}
\partial_{\alpha}R(\mathbf{x})=R_0\partial_{\alpha}\big(\exp(x_\beta E_\beta)\big), \quad \mbox{where}~\partial_\alpha=\partial_{x_\alpha}. 
\end{equation}
We set
\begin{equation} \label{B-4-1}
\theta=\|\mathbf{x}\| ,\quad \mathbf{v}=\frac{\mathbf{x}}{\theta }.
\end{equation}
Then, we use \eqref{eqn:tvtoR} to find
\[
\exp(x_\beta E_\beta) =\exp(\theta \widehat{\mathbf{v}}) =I+\frac{\sin\theta }{\theta } \widehat{\mathbf{x}}+\frac{1-\cos\theta }{\theta ^2}\widehat{\mathbf{x}}^2.
\]
This yields,
\begin{align}
\begin{aligned}\label{E-5}
\partial_{\alpha}\big(\exp(x_\beta E_\beta)\big)&=\left(\frac{\theta \cos\theta -\sin\theta}{\theta ^2}\right)(\partial_\alpha \theta )\widehat{\mathbf{x}}+\frac{\sin\theta }{\theta }\partial_\alpha\widehat{\mathbf{x}} \\
&\quad +\left(\frac{\theta \sin\theta -2(1-\cos\theta )}{\theta ^3}\right)(\partial_\alpha \theta )\widehat{\mathbf{x}}^2+\frac{1-\cos\theta }{\theta ^2}\partial_\alpha \widehat{\mathbf{x}}^2.
\end{aligned}
\end{align}
\begin{lemma} \label{L2.1}
For any $\bx \in \bbr^3$, one has 
\begin{equation}
\label{E-6}
\partial_\alpha\widehat{\mathbf{x}}=E_\alpha, \qquad \partial_\alpha\widehat{\mathbf{x}}^2=\{E_\alpha, \widehat{\mathbf{x}}\},
\end{equation}
where $\{\cdot,\cdot\}$ denotes the anticommutator $\{A,B\} :=AB+BA$.
\end{lemma}
\begin{proof}
The first identity follows from \eqref{eqn:hat} and \eqref{B-4-0-0} directly, and the second identity can be derived as follows.
\[
\partial_\alpha\widehat{\mathbf{x}}^2 =\partial_\alpha\left(x_\beta x_\gamma E_\beta E_\gamma\right)  =x_\beta E_\beta E_\alpha+x_\gamma E_\alpha E_\gamma=x_\beta\{E_\alpha, E_\beta\} = \{E_\alpha, x_\beta E_\beta\}  =  \{E_\alpha, \widehat{\mathbf{x}}\}.
\]
\end{proof}
\begin{remark}  \label{R2.1} As an application of Lemma \ref{L2.1}, we can check that $\partial_\alpha R(\bx)$ given by  \eqref{eqn:pR-alpha} lies in the tangent space of $\mathrm{SO}(3)$ at $R(\bx)$. Indeed, use \eqref{E-6} and $\partial_\alpha\theta =\frac{x_\alpha}{\theta }$ in \eqref{B-4-1} to rewrite \eqref{E-5} as:
\begin{align}
\begin{aligned}\label{B-6-1}
\partial_{\alpha}\big(\exp(x_\beta E_\beta)\big)&=\left(\frac{\theta \cos\theta -\sin\theta}{\theta ^3}\right) x_\alpha \widehat{\mathbf{x}}+\frac{\sin\theta }{\theta} E_\alpha \\
&\quad +\left(\frac{\theta \sin\theta -2(1-\cos\theta )}{\theta ^4}\right) x_\alpha \widehat{\mathbf{x}}^2+\frac{1-\cos\theta }{\theta ^2} \{E_\alpha, \widehat{\mathbf{x}}\}.
\end{aligned}
\end{align}
By a straightforward but tedious calculation, one can show that $\big(\exp(x_\gamma E_\gamma)\big)^\top \partial_{\alpha}\big(\exp(x_\beta E_\beta)\big)$
is a skew-symmetric matrix. This implies 
\[ \partial_\alpha R(\bx) \in T_{R(\bx)} \mathrm{SO}(3). \]
\end{remark}
In the following lemma, we list several other elementary results to be used later.
\begin{lemma}\label{L2.2}
For $\bx \in \bbr^3$, let $\widehat{\mathbf{x}}$ be given by  \eqref{B-4-0} and $\theta = \|\bx\|$. Then the following identities hold:
\[
\widehat{\mathbf{x}}^3=-\theta ^2\widehat{\mathbf{x}} \qquad \mbox{and} \qquad 
x_\alpha \widehat{\mathbf{x}}^2+\widehat{\mathbf{x}}^2E_\alpha \widehat{\mathbf{x}}=0, \qquad \text{ for all } \alpha = 1,2,3.
\]
\end{lemma}
\begin{proof} 
See Appendix \ref{appendix:L2.2}.
\end{proof}


\subsection{The metric tensor} \label{sec:2.3}
In this subsection, we derive explicit representations for the metric tensor $g_{\alpha \beta}(\bx)$ and its inverse $g^{\alpha \beta}(\bx)$.  First, we study a set of elementary estimates in the following lemma. 
\begin{lemma}\label{L2.3}
For $\bx \in \bbr^3$, let $\widehat{\mathbf{x}}$ be given by  \eqref{B-4-0} and $\theta = \|\bx\|$. Then, the following identities hold 
\begin{eqnarray*}
&& (i)~\langle E_\alpha, E_\beta\rangle_{\mathrm{F}}=2\delta_{\alpha \beta}, \qquad
\langle\widehat{\mathbf{x}},\widehat{\mathbf{x}}\rangle_{\mathrm{F}} =2\theta ^2, \qquad
\langle \widehat{\mathbf{x}}, E_\alpha\rangle_{\mathrm{F}}=2x_\alpha, \\[5pt]
&& (ii)~\langle \widehat{\mathbf{x}}^2,\widehat{\mathbf{x}}^2\rangle_{\mathrm{F}} =2\theta ^4, \quad
\langle \widehat{\mathbf{x}}^2, \{E_\alpha, \widehat{\mathbf{x}}\}\rangle_{\mathrm{F}} =4 x_\alpha\theta ^2,\quad
\langle \{E_\alpha, \widehat{\mathbf{x}}\}, \{E_\beta, \widehat{\mathbf{x}}\}\rangle_{\mathrm{F}} =6x_\alpha x_\beta+2\delta_{\alpha\beta}\theta ^2.
\end{eqnarray*}
\end{lemma}
\begin{proof}
(i)~By direct calculation, one can find
\[ \langle E_\alpha, E_\beta\rangle_{\mathrm{F}}=2\delta_{\alpha\beta}. \]
Then, one has 
\[
\langle\widehat{\mathbf{x}},\widehat{\mathbf{x}}\rangle_{\mathrm{F}}= x_\alpha x_\beta\langle E_\alpha, E_\beta\rangle_{\mathrm{F}}=2x_\alpha x_\alpha=2\theta ^2,
\]
and
\[
\langle \widehat{\mathbf{x}}, E_\alpha\rangle_{\mathrm{F}}= \langle x_\beta E_\beta, E_\alpha \rangle_{\mathrm{F}}=x_\beta \langle E_\beta, E_\alpha\rangle_{\mathrm{F}}=2x_\alpha.
\]
(ii) Since the derivations for identities in (ii) are rather lengthy, we present them in Appendix  \ref{appendix:L2.3}.
\end{proof}

\begin{proposition} \label{P2.2}
For $\bx \in \bbr^3$, let $\widehat{\mathbf{x}}$ be given by  \eqref{B-4-0} and $\theta = \|\bx\|$. Then, the metric tensor and its inverse are given as follows.
\begin{align*}
\begin{aligned}
& (i)~g_{\alpha \beta}(\mathbf{x}) =\left(\frac{2\cos\theta -2+\theta ^2}{\theta ^4}\right) x_\alpha x_\beta  +\frac{2(1-\cos\theta )}{\theta ^2}\delta_{\alpha \beta }, \\
& (ii)~g^{\alpha\beta}(\mathbf{x}) =\left( \frac{2\cos\theta -2+\theta ^2}{2\theta ^2(\cos\theta -1)}\right) x_\alpha x_\beta+\frac{\theta ^2}{2(1-\cos\theta )} \delta^{\alpha\beta}.
\end{aligned}
\end{align*}
\end{proposition}
\begin{proof}
(i)~It follows from \eqref{eqn:Rmetric} and 
\[
 \partial_{\alpha}R(\mathbf{x})=R_0\partial_{\alpha}\big(\exp(x_\beta E_\beta)\big), \quad \mbox{where}~\partial_\alpha=\partial_{x_\alpha},
\]
that
\begin{align}
\begin{aligned} \label{NNN-1}
g_{\alpha \beta}(\bx) &=g(\partial_\alpha R(\mathbf{x}), \partial_\beta R(\mathbf{x})) = \frac{1}{2} \Big \langle\partial_\alpha R(\mathbf{x}), \partial_\beta R(\mathbf{x}) \Big \rangle_{\mathrm{F}}  \\
&= \frac{1}{2} \Big \langle \partial_{\alpha}\big(\exp(x_\beta E_\beta)\big), \partial_{\alpha}\big(\exp(x_\beta E_\beta)\big) \Big \rangle_{\mathrm{F}}
\end{aligned}
\end{align}
On the other hand, since  $E_\alpha$ and $\widehat{\mathbf{x}}$ are skew-symmetric, $\{E_\alpha, \widehat{\mathbf{x}}\}$ is also symmetric. Now, we decompose  \eqref{B-6-1} in skew-symmetric and symmetric parts:
\begin{equation} \label{NNN-2}
\partial_{\alpha}\big(\exp(x_\beta E_\beta)\big)=A_\alpha+S_\alpha,
\end{equation}
where $A_\alpha$ and  $S_\alpha$ are skew-symmetric and symmetric matrices, respectively. One finds:
\begin{align*}
A_\alpha&=\left(\frac{ \theta \cos\theta -\sin\theta }{\theta ^3}\right)x_\alpha\widehat{\mathbf{x}}+\frac{\sin\theta }{\theta } E_\alpha,\\[5pt]
S_\alpha&=\left(\frac{\theta \sin\theta -2(1-\cos\theta )}{\theta ^4}\right)x_\alpha\widehat{\mathbf{x}}^2+\frac{1-\cos\theta }{\theta ^2}\{E_\alpha, \widehat{\mathbf{x}}\}.
\end{align*}
Next, we combine \eqref{NNN-1} and \eqref{NNN-2} to find
\begin{equation} \label{eqn:gab}
g_{\alpha\beta}(\bx) =\frac{1}{2} \langle\partial_\alpha R(\mathbf{x}), \partial_\beta R(\mathbf{x})\rangle_{\mathrm{F}} =\frac{1}{2} \langle A_\alpha, A_\beta\rangle_{\mathrm{F}}+ \frac{1}{2}\langle S_\alpha, S_\beta\rangle_{\mathrm{F}},
\end{equation}
where we used that $\langle A, S\rangle_{\mathrm{F}}=0$ for any skew-symmetric matrix $A$ and symmetric matrix $S$.  For the estimation of the two terms in the R.H.S. of \eqref{eqn:gab}, we use Lemma \ref{L2.3}.  \newline

\noindent $\bullet$~Estimate of $\langle A_\alpha, A_\beta\rangle_{\mathrm{F}}$:  We use Lemma \ref{L2.3} (i) to find:
\begin{align}
\begin{aligned} \label{eqn:A-prod}
&\langle A_\alpha, A_\beta\rangle_{\mathrm{F}} \\
&=\left(\frac{\theta \cos\theta -\sin\theta}{\theta ^3}\right)^2x_\alpha x_\beta\langle \widehat{\mathbf{x}}, \widehat{\mathbf{x}}\rangle_{\mathrm{F}}  \\
&\quad+\left(\frac{\theta \cos\theta -\sin\theta}{\theta ^3}\right) \frac{\sin\theta }{\theta }\big(x_\alpha\langle \widehat{\mathbf{x}}, E_\beta\rangle_{\mathrm{F}}+x_\beta\langle \widehat{\mathbf{x}}, E_\alpha\rangle_{\mathrm{F}} \big) +\left(\frac{\sin\theta }{\theta }\right)^2\langle E_\alpha, E_\beta\rangle_{\mathrm{F}} \\
&=2\left(\frac{\theta \cos\theta -\sin\theta}{\theta ^2}\right)^2x_\alpha x_\beta
+4\left(\frac{\theta \cos\theta -\sin\theta}{\theta ^3}\right)\frac{\sin\theta }{\theta } x_\alpha x_\beta
+2\left(\frac{\sin\theta }{\theta }\right)^2\delta_{\alpha \beta} \\
& = 2\left(\frac{\theta^2 \cos^2 \theta - \sin^2 \theta}{\theta^4}\right) x_\alpha x_\beta + 2\left(\frac{\sin\theta }{\theta }\right)^2\delta_{\alpha \beta}.
\end{aligned}
\end{align}

\vspace{0.2cm}

\noindent $\bullet$~Estimate of $\langle S_\alpha, S_\beta\rangle_{\mathrm{F}}$:~We use Lemma \ref{L2.3} (ii)  to find:
\begin{align*}
\begin{aligned}
&\langle S_\alpha, S_\beta\rangle_{\mathrm{F}} \\
&=\left(\frac{\theta \sin\theta -2(1-\cos\theta )}{\theta ^4}\right)^2x_\alpha x_\beta \langle\widehat{\mathbf{x}}^2,\widehat{\mathbf{x}}^2\rangle_{\mathrm{F}} +\left(\frac{1-\cos\theta }{\theta ^2}\right)^2\langle \{E_\alpha, \widehat{\mathbf{x}}\}, \{E_\beta, \widehat{\mathbf{x}}\}\rangle_{\mathrm{F}}\\
&\quad +\left(\frac{\theta \sin\theta -2(1-\cos\theta )}{\theta ^4}\right)\frac{1-\cos\theta }{\theta ^2}(x_\alpha\langle \widehat{\mathbf{x}}^2, \{E_\beta, \widehat{\mathbf{x}}\}\rangle_{\mathrm{F}}+x_\beta\langle \widehat{\mathbf{x}}^2, \{E_\alpha, \widehat{\mathbf{x}}\}\rangle_{\mathrm{F}})\\
&=2\left(\frac{\theta \sin\theta -2(1-\cos\theta )}{\theta ^2}\right)^2x_\alpha x_\beta + \left(\frac{1-\cos\theta }{\theta ^2}\right)^2(6x_\alpha x_\beta+2\delta_{\alpha \beta}\theta ^2) \\
& \quad + \left(\frac{\theta \sin\theta -2(1-\cos\theta )}{\theta ^2}\right)\frac{1-\cos\theta }{\theta ^2}(8x_\alpha x_\beta).
\end{aligned}
\end{align*}
By grouping similar terms, one has 
\begin{equation}
\label{eqn:S-prod}
\langle S_\alpha, S_\beta\rangle_{\mathrm{F}} = 2\left(\frac{\theta^2 \sin^2 \theta - (1-\cos \theta)^2}{\theta^4}\right) x_\alpha x_\beta + 2\, \frac{(1-\cos \theta)^2}{\theta^2} \delta_{\alpha \beta}
\end{equation}
Finally, we combine \eqref{eqn:gab}, \eqref{eqn:A-prod} and \eqref{eqn:S-prod} to obtain 
\begin{equation}
\label{eqn:metric}
g_{\alpha \beta}(\mathbf{x})=\left(\frac{2\cos\theta -2+\theta ^2}{\theta ^4}\right) x_\alpha x_\beta  +\frac{2(1-\cos\theta )}{\theta ^2}\delta_{\alpha \beta }.
\end{equation}
(ii)~We postpone the calculation of the inverse of $(g_{\alpha \beta}(\bx))$ to Appendix \ref{appendix:metric-inv}. 
\end{proof}

\subsection{Christoffel symbols and geodesic equations.} \label{sec:2.4}
A direct calculation in Appendix \ref{appendix:Christoffel} and formula for the Christoffel symbols
\[
\Gamma^\gamma_{\alpha \beta}=\frac{1}{2}g^{\gamma \kappa}(\partial_\beta g_{\kappa \alpha}+\partial_\alpha g_{\kappa \beta}-\partial_\kappa g_{\alpha \beta}),
\]
yield
\begin{align}
\begin{aligned}
\label{eqn:Christoffel}
\Gamma^\gamma_{\alpha \beta}(\bx)&=  \frac{(\sin \theta + \theta)(\cos \theta -1) + \theta^2 \sin \theta}{\theta^5(\cos \theta -1)}  x_\alpha x_\beta x_\gamma \\[3pt]
&\quad + \frac{2\cos\theta -2+\theta \sin\theta}{2 \theta^2 (1-\cos\theta )}(x_\alpha\delta_\beta^\gamma+x_\beta\delta_\alpha^\gamma) - \frac{\sin\theta -\theta}{\theta ^3} \delta_{\alpha\beta}x_\gamma.
\end{aligned}
\end{align}
One advantage of working with exponential coordinates is that geodesics are images of straight lines under the parametrization map. To check the expressions \eqref{eqn:Christoffel}, it can be shown that for a fixed vector $\bu \in \bbr^3$, the straight line $\bx(t) = t \bu$ satisfies the geodesic equations
\[
\frac{d^2 x_\gamma}{dt^2}+\Gamma^\gamma_{\alpha \beta} (\bx(t))\frac{d x_\alpha}{dt}\frac{d x_\beta}{dt}=0, 
\]
for all $\gamma=1,2,3$. The second derivative is zero, so we only need to check 
\begin{equation}
\label{eqn:geod-check}
\Gamma^\gamma_{\alpha \beta} (\bx(t)) u_\alpha u_\beta = 0.
\end{equation}
Detailed arguments will be given in Appendix \ref{appendix:geod}.
 
 
\section{Parallel transport on $\mathrm{SO}(3)$} \label{sec:3}
\setcounter{equation}{0}
 
In this section we derive an explicit formula for the parallel transport on $\mathrm{SO}(3)$ along a geodesic, and provide its geometric interpretation. 


\subsection{Parallel transport} \label{subsect:pt-coord} Consider a geodesic curve on $\mathrm{SO}(3)$ given in exponential coordinates by $x_\alpha (t)=t u_\alpha$, for a fixed $\bu =(u_1,u_2,u_3) \in \bbr^3$, and the parallel transport along this geodesic of a vector of components $v_\alpha(t)$.  Then, the equations for parallel transport are given by
\[
\frac{dv_\gamma}{dt}+\Gamma_{\alpha \beta}^\gamma(\bx(t)) \dot{x}_\alpha v_\beta=0, \qquad \gamma=1,2,3.
\]
Since $\dot{x}_\alpha = u_\alpha$ along the geodesic, we can write the equations as
\begin{equation}
\label{eqn:pt-comp}
\dot{v}_\gamma+\Gamma_{\alpha \beta}^\gamma(\bx(t)) u_\alpha v_\beta=0, \qquad \gamma = 1,2,3.
\end{equation}
 Next, we calculate $\Gamma^\gamma_{\alpha \beta}(\bx(t)) u_\alpha$ with the Christoffel symbols given by \eqref{eqn:Christoffel}. By direct substitution of $\bx(t)=t \bu$ in \eqref{eqn:Christoffel} (in particular, use $\|\bx(t)\| = \|\bu\| t$), we obtain
\begin{align}
\begin{aligned} \label{AC-1}
&\Gamma^\gamma_{\alpha \beta}(\bx(t)) u_\alpha \\
& \hspace{0.2cm} =  \frac{(\sin (\|\bu\| t) +  t \|\bu\|)(\cos  (\|\bu\| t) -1) + t^2 \|\bu\|^2 \sin  (\|\bu\| t)}{t^5 \|\bu\|^5 (\cos  (\|\bu\| t) -1)}  t^3 u_\alpha u_\beta u_\gamma u_\alpha\\[3pt]
& \hspace{0.2cm}  + \frac{2\cos  (\|\bu\| t) -2+ t \|\bu\| \sin  (\|\bu\| t)}{2 t^2 \|\bu\|^2 (1-\cos  (\|\bu\| t) )}(t u_\alpha\delta_\beta^\gamma+t u_\beta\delta_\alpha^\gamma) u_\alpha- \frac{\sin  (\|\bu\| t) - t\|\bu\|}{t^3 \|\bu\|^3} \delta_{\alpha\beta} t u_\gamma u_\alpha \\[3pt]
& \hspace{0.2cm} = \frac{2\cos(\|\bu\| t)-2+t \|\bu\|\sin(\|\bu\| t)}{2t(1-\cos(\|\bu\| t))} \left(\delta_\beta^\gamma-\frac{u_\beta u_\gamma}{\| \bu \|^2}\right),
\end{aligned}
\end{align}
where the second equal sign follows from elementary algebra and $u_\alpha u_\alpha = \|\bu\|^2$. 

For calculations in this section we reset the notation for $\theta$ used in Section \ref{sec:3}, and denote:
\[
\theta:=\|\mathbf{u}\|,\quad \mathbf{n}:=\frac{\mathbf{u}}{\theta}.
\]
With this notation, rewrite \eqref{AC-1} as
\[
\Gamma_{\alpha \beta}^\gamma (\bx(t) )u_\alpha=\frac{2\cos(\theta t)-2+(\theta t)\sin(\theta t)}{2t(1-\cos(\theta t))}\left(\delta_\beta^\gamma-\frac{u_\beta u_\gamma}{\theta^2}\right).
\]
Then, we use the above expression in the equations \eqref{eqn:pt-comp} for parallel transport to find
\begin{align}
\begin{aligned}\label{B-9}
0&=\dot{v}_\gamma+\left(\frac{2\cos(\theta t)-2+(\theta t)\sin(\theta t)}{2t(1-\cos(\theta t))}\right)\left(\delta_\beta^\gamma-\frac{u_\beta u_\gamma}{\theta^2}\right)v_\beta\\
&=\dot{v}_\gamma+\left(\frac{2\cos(\theta t)-2+(\theta t)\sin(\theta t)}{2t(1-\cos(\theta t))}\right)\left(v_\gamma-\frac{u_\gamma}{\theta^2}(u_\beta v_\beta)\right).
\end{aligned}
\end{align}

Multiply both sides of \eqref{B-9} by $u_\gamma$ and sum  up the resulting relation over $\gamma$ to obtain
\begin{align*}
0&=\dot{v}_\gamma u_\gamma+\left(\frac{2\cos(\theta t)-2+(\theta t)\sin(\theta t)}{2t(1-\cos(\theta t))}\right) \underbrace{\left(v_\gamma u_\gamma-\frac{u_\gamma u_\gamma}{\theta^2}(u_\beta v_\beta)\right)}_{=0}.
\end{align*}
The second term in the R.H.S. above vanishes, and the vector $\bu$ is constant. This leads to 
\[ \frac{d}{dt}(v_\gamma u_\gamma) =0, \]
and hence $v_\gamma(t)u_\gamma$ is a conserved quantity:
\begin{equation}
\label{eqn:calC}
\mathcal{C}=v_\gamma(t)u_\gamma=v_\gamma(0)u_\gamma,
\end{equation}
for a constant $\calC$. The conservation property in \eqref{B-9} implies
\begin{align}\label{B-10}
0=\dot{v}_\gamma(t)+\left(\frac{2\cos(\theta t)-2+(\theta t)\sin(\theta t)}{2t(1-\cos(\theta t))}\right)\left(v_\gamma(t)-\frac{\mathcal{C}u_\gamma}{\theta^2}\right).
\end{align}

To find the exact solution of \eqref{B-10}, we define the function $f$ by
\begin{equation}
\label{eqn:f}
f(t) :=\begin{cases}
\displaystyle\sqrt{\frac{1-\cos(\theta t)}{t^2}},\quad&\text{if}\quad t>0, \\[5pt]
\displaystyle\frac{\theta}{\sqrt{2}},&\text{if}\quad t=0.
\end{cases}
\end{equation}
Here we can easily check that $f$ is a $C^1$ function defined on $[0, \infty)$, and $f$ satisfies
\[
\frac{f'(t)}{f(t)}=\frac{2\cos(\theta t)-2+(\theta t)\sin(\theta t)}{2t(1-\cos(\theta t)},
\]
so one can write \eqref{B-10} as
\[
0=f(t)\dot{v}_\gamma(t)+f'(t)\left(v_\gamma(t)-\frac{\mathcal{C}u_\gamma}{\theta^2}\right).
\]
Furthermore, we have
\[
0=\frac{d}{dt}\left(f(t)v_\gamma(t)-\frac{\mathcal{C}u_\gamma f(t)}{\theta^2}\right).
\]
This implies 
\[
f(t)v_\gamma(t)-\frac{\mathcal{C}u_\gamma f(t)}{\theta^2}=\lim_{\tau\to 0}\left(f(\tau)v_\gamma(\tau)-\frac{\mathcal{C}u_\gamma f(\tau)}{\theta^2}\right)=\frac{\theta}{\sqrt{2}}v_\gamma(0)-\frac{\mathcal{C} u_\gamma}{\sqrt{2}\theta},
\]
which yields the explicit solution of the parallel transport equations \eqref{eqn:pt-comp}:
\begin{equation}
\label{eqn:pt-sol}
v_\gamma(t)=\frac{1}{f(t)}\left(\frac{\theta}{\sqrt{2}}v_\gamma(0)-\frac{\mathcal{C} u_\gamma}{\sqrt{2}\theta}\right)+\frac{\mathcal{C}u_\gamma}{\theta^2},
\end{equation} 
with $\theta = \|\bu \|$ and $\mathcal{C}$ given by \eqref{eqn:calC}.


\subsection{Coordinate-free closed form expression} \label{subsect:pt-coord-free} We use the explicit formula \eqref{eqn:pt-sol} to derive a coordinate free expression for parallel transport along a geodesic. We fix $R_0,R_1 \in \mathrm{SO}(3)$, with $d(R_0,R_1) <\pi$, and consider the unique geodesic path between them. We set exponential coordinates centered at $R_0$, and also set $\bu \in \bbr^3$ by
\[
\hbu = \log(R_0^\top  R_1).
\]
Equivalently one has 
\[ \log_{R_0}R_1= R_0 \hbu. \]
The geodesic curve between the two matrices is given by 
\[ \calR:[0,1] \to \mathrm{SO}(3), \quad  \calR(t) = R_0 \exp(t \hbu). \]
 In particular, $\theta = \| \bu\|$ represents the geodesic distance between $R_0$ and $R_1$ (see \eqref{eqn:geod-dist}):
\[
 \theta =  d(R_0,R_1) = \operatorname{acos}\left(\frac{\operatorname{tr}(R_0^\top  R_1) -1}{2} \right).
\]
Also, by \eqref{eqn:Rtotv}, the unit vector $\bn=\frac{\bu}{\theta}$ satisfies:
\[
\hbn = \frac{1}{2\sin\theta}(R_0^\top  R_1-R_1^\top  R_0).
\]

For notational convenience, we adopt the following notation for the tangent vectors in exponential coordinates (see \eqref{eqn:pR-alpha} and \eqref{B-6-1}):
\[
X_\alpha(R(\bx)) := \partial_\alpha R(\bx).
\]
Then, it follows from \eqref{eqn:pR-alpha} and \eqref{B-6-1} that 
\begin{align} 
\begin{aligned} \label{eqn:XalphaR1}
& X_\alpha(R_0)=R_0E_\alpha \quad \mbox{and} \\[3pt]
& X_\alpha(R_1)  =R_0\partial_\alpha \left(\exp\left(x_\beta E_\beta\right)\right)\big|_{\mathbf{x}=\mathbf{u}}\\[3pt]
& \hspace{0.2cm}=\left(\frac{-\sin\theta+\theta\cos\theta}{\theta^3}\right)u_\alpha R_0\widehat{\mathbf{u}}+\frac{\sin\theta}{\theta} R_0 E_\alpha +\left(\frac{\theta\sin\theta-2(1-\cos\theta)}{\theta^4}\right)u_\alpha R_0\widehat{\mathbf{u}}^2 \\
& \hspace{0.2cm} +\frac{1-\cos\theta}{\theta^2}R_0\{E_\alpha, \widehat{\mathbf{u}}\}.
\end{aligned}
\end{align}

Now take a tangent vector $V(0) \in T_{R_0}  \mathrm{SO}(3)$ given in components by
\begin{equation}
\label{eqn:V0}
V(0)=v_\alpha(0)X_\alpha(R_0).
\end{equation}
According to the calculations performed in Section \ref{subsect:pt-coord} (see equation \eqref{eqn:pt-sol}), the parallel transport of $V(0)$ from $R_0$ to $R_1$ along the geodesic curve is  
\begin{equation}
\label{eqn:V1}
V(1)=v_\alpha(1)X_\alpha(R_1),
\end{equation}
where
\[
v_\alpha(1)=\frac{1}{f(1)}\left(\frac{\theta}{\sqrt{2}}v_\alpha(0)-\frac{\mathcal{C} u_\alpha}{\sqrt{2}\theta}\right)+\frac{\mathcal{C}u_\alpha}{\theta^2}.
\]
Here, $\calC$ is the constant of motion given by \eqref{eqn:calC}. Using  \eqref{eqn:f} and some simple manipulations we write the components of $V(1)$ as
\begin{align}
\label{eqn:v1-mod}
v_\alpha(1) =\sqrt{\frac{1}{1-\cos\theta}}\left(\frac{\theta}{\sqrt{2}}v_\alpha(0)-\frac{\mathcal{C} u_\alpha}{\sqrt{2}\theta}\right)+\frac{\mathcal{C}u_\alpha}{\theta^2} =\frac{\theta v_\alpha(0)}{2\sin\frac{\theta }{2}}+\mathcal{C}u_\alpha\left(\frac{1}{\theta^2}-\frac{1}{2\theta\sin\frac{\theta}{2}}\right).
\end{align}

One can then use $\eqref{eqn:XalphaR1}_2$ and \eqref{eqn:v1-mod} in \eqref{eqn:V1} (see calculations in Appendix \ref{appendix:pt-closed}) to find that the parallel-transported vector $V(1)$ has the following coordinate-free expression:
\begin{align}
\begin{aligned}\label{B-11}
V(1)&=\mathcal{C}R_0\widehat{\mathbf{u}}\left(\frac{\cos\theta-\cos\frac{\theta}{2}}{\theta^2} \right)+\mathcal{C}R_0\widehat{\mathbf{u}}^2\left(\frac{\sin\theta-2\sin\frac{\theta}{2}}{\theta^3}\right)\\
&\quad +\frac{\sin\frac{\theta}{2}}{\theta}\left(V(0)\widehat{\mathbf{u}}+R_0\widehat{\mathbf{u}}R_0^\top V(0)\right)+\cos\frac{\theta }{2}V(0).
\end{aligned}
\end{align}


\subsection{A geometric interpretation}\label{sec:3.3}
By \eqref{eqn:tspace}, the tangent vectors $V(0)$ and $V(1)$ take the following form: 
\[
V(0) = R_0 A_0, \qquad \text{ and } \qquad V(1) = R_1 A_1,
\]
where $A_0$ and $A_1$ are skew-symmetric matrices.  To provide better insight, which leads to a nice geometric interpretation of parallel transport, we now express \eqref{B-11} in terms of these skew-symmetric matrices. 

Therefore, we set
\begin{equation}
\label{eqn:A0A1}
A_0=R_0^\top V(0), \qquad A_1=R_1^\top V(1),
\end{equation}
and also define $\veeA_0$ and $\veeA_1$ such that $\widehat{\veeA}_0=A_0$ and $\widehat{\veeA}_1=A_1$
For convenience of notation, denote
\[
R=R_0^\top R_1.
\]
Then, it follows from \eqref{B-11} that 
\begin{align}
\label{eqn:A1-int}
\begin{aligned}
A_1&=\mathcal{C}R^\top \widehat{\mathbf{u}}\left(\frac{\cos\theta-\cos\frac{\theta}{2}}{\theta^2} \right)+\mathcal{C}R^\top \widehat{\mathbf{u}}^2\left(\frac{\sin\theta-2\sin\frac{\theta}{2}}{\theta^3}\right)\\
&\quad +\frac{\sin\frac{\theta}{2}}{\theta}R^\top \left(A_0\widehat{\mathbf{u}}+\widehat{\mathbf{u}}A_0\right)+\cos\frac{\theta }{2}R^\top A_0.
\end{aligned}
\end{align}
We take  the transpose in Rodrigues's formula \eqref{eqn:tvtoR} (applied to $R=R_0^\top  R_1$) to get 
\begin{equation}
\label{eqn:RT}
R^\top =I-\frac{\sin\theta}{\theta}\widehat{\mathbf{u}}+\frac{1-\cos\theta}{\theta^2}\widehat{\mathbf{u}}^2.
\end{equation}
Hence, one has
\begin{equation}
\label{eqn:RT-cont}
R^\top \widehat{\mathbf{u}}=\cos\theta \widehat{\mathbf{u}}-\frac{\sin\theta}{\theta}\widehat{\mathbf{u}}^2, \qquad \text{ and } \qquad
R^\top \widehat{\mathbf{u}}^2=\cos\theta \widehat{\mathbf{u}}^2+\theta\sin\theta\widehat{\mathbf{u}},
\end{equation}
where we used that $\widehat{\mathbf{u}}^3=-\theta^2\widehat{\mathbf{u}}$ (see Lemma \ref{L2.1}).  

By using \eqref{eqn:RT} and  \eqref{eqn:RT-cont} in \eqref{eqn:A1-int}, after some manipulations (see Appendix \ref{appendix:A1}), we derive the following expression for $A_1$:
\begin{align}
\begin{aligned} \label{B-12}
A_1&=\mathcal{C}\left(\frac{1-\cos\theta\cos\frac{\theta}{2}-2\sin\theta\sin\frac{\theta}{2}}{\theta^2}\right)\widehat{\mathbf{u}}
+\frac{\sin\frac{\theta}{2}}{\theta}\left(A_0\widehat{\mathbf{u}}-\widehat{\mathbf{u}}A_0\right) \\
&\quad -\frac{2\sin^2\frac{\theta}{2}\cos\frac{\theta}{2}}{\theta^2}\widehat{\mathbf{u}}A_0\widehat{\mathbf{u}}
+\cos\frac{\theta }{2}A_0.
\end{aligned}
\end{align}

The expression for $A_1$ can be further simplified using the following lemma.
\begin{lemma}
\label{lem:various}
Let $\bx, \by \in \bbr^3$ and their corresponding skew-symmetric matrices $\hbx,\hby \in \mathfrak{so}(3)$. The following identities hold:
\begin{equation} \label{eqn:id1}
\Tr(\hbx \hby) = - 2 \, \bx \cdot \by, \qquad \hbx \hby - \hby \hbx = \widehat{\bx \times \by}, \qquad \hbx \hby \hbx = - (\bx \cdot \by) \hbx.
\end{equation}
\end{lemma}
\begin{proof}
See Appendix \ref{proofL:various}.
\end{proof}


First note that by $\eqref{eqn:XalphaR1}_1$, \eqref{eqn:V0} and \eqref{eqn:A0A1}, we have:
\[
A_0 = v_\alpha(0) R_0^\top X_\alpha(R_0) = v_\alpha(0) E_\alpha,
\]
and hence, by \eqref{eqn:calC},
\[
\mathcal{C} = u _\alpha v_\alpha(0) =  \frac{1}{2} \langle \hbu, A_0 \rangle_{\mathrm{F}}.
\]
Then, it follows from $\eqref{eqn:id1}_1$ and $\eqref{eqn:id1}_3$ that 
\[
\calC = -\frac{1}{2} \Tr(\hbu A_0) = \bu \cdot \veeA_0, \qquad \text{ and } \qquad
\hbu A_0 \hbu = - ( \bu \cdot \veeA_0) \bu.
\]
Using the equations above, combine the first and third terms in the R.H.S. of \eqref{B-12} to get
\begin{align}
\begin{aligned} \label{B-12-mod}
A_1&= \mathcal{C}\left(\frac{1-\cos\theta\cos\frac{\theta}{2}-\sin\theta\sin\frac{\theta}{2}}{\theta^2}\right)\widehat{\mathbf{u}}
+\frac{\sin\frac{\theta}{2}}{\theta}\left(A_0\widehat{\mathbf{u}}-\widehat{\mathbf{u}}A_0\right)
+\cos\frac{\theta }{2}A_0 \\
&= (\bn \cdot \veeA_0) \left(1-\cos\frac{\theta}{2}\right)\widehat{\bn}
+\sin\frac{\theta}{2} \left(A_0 \hbn-\hbn A_0\right)
+\cos\frac{\theta }{2}A_0,
\end{aligned}
\end{align}
where we also used
\[
\cos\theta\cos\frac{\theta}{2}+\sin\theta\sin\frac{\theta}{2}= \cos \frac{\theta}{2}, \quad \text{ and } \quad \bn = \frac{\bu}{\theta}.
\]

Using $\eqref{eqn:id1}_2$, we can write \eqref{B-12-mod} as:
\begin{equation}
\label{B-12-simple}
A_1 =\left(1-\cos\frac{\theta}{2}\right)  (\bn \cdot \veeA_0)\, \widehat{\mathbf{n}}
+\sin\frac{\theta}{2}\,  \widehat{\veeA_0 \times \bn} +\cos\frac{\theta }{2} A_0.
\end{equation}
By applying the vee operator to \eqref{B-12-simple} we then find:
\begin{equation}
\label{B-12-geom}
\veeA_1=\left(1-\cos\frac{\theta}{2}\right)  (\bn \cdot \veeA_0)\, \mathbf{n}
+\sin\frac{\theta}{2}\, \veeA_0 \times \bn +\cos\frac{\theta }{2} \, \veeA_0.
\end{equation}

\begin{remark} 
\label{rmk:pt-geom}Remarkably, \eqref{B-12-geom} is in the form of Rodrigues' rotation formula, with $\veeA_1$ representing the rotation (according to the right-hand-rule) of $\veeA_0$ about the axis $-\bn$ by an angle $\frac{\theta}{2}$. For the case in which $\veeA_0$ is in the direction of $\bn$, we have $ \veeA_1=\veeA_0, $ as expected.
\end{remark}
\medskip

{\em Parallel transport by general theory on Lie groups.} In \cite[Chapter 2]{He}, listed as an exercise, one can find a general expression for parallel transport along geodesics on connected Lie groups. Specifically, consider a connected Lie group $G$ with Lie algebra $\mathfrak{g}$, endowed with an affine connection that is invariant under left and right translations and under the map $g \mapsto g^{-1}$. For $g\in G$, denote by $L_g$ and $R_g$ the left and right translations by $g$, respectively. Take $X,Y \in \mathfrak{g}$. The parallel translate of $X$ along the geodesic $\exp(t Y)$, $0 \leq t \leq 1$, is given by:
\[
DL_{\exp \left(\frac{1}{2}Y\right)} DR_{\exp \left(\frac{1}{2}Y \right)} X.
\]

The rotation group and its connection satisfy the assumptions to apply the formula for parallel transport above. Adapting it to the setup of $\mathrm{SO}(3)$, we can find that the parallel transport of $A_0 \in \mathfrak{so(3)} = T_I \mathrm{SO}(3)$ from $R_0=I$ to $R_1 = \exp(\hbu)$ is given by:
\[
R_1 A_1 = \exp\left(\frac{1}{2} \hbu\right) A_0  \exp\left(\frac{1}{2} \hbu\right).
\]
This is equivalent to
\[
A_1 = \exp\left(-\frac{1}{2} \hbu\right) A_0  \exp\left(\frac{1}{2} \hbu\right),
\]
or using Rodrigues' formula:
\begin{equation}
\label{eqn:He-1}
A_1 = \left( I - \sin \left(\frac{\theta}{2}\right) \hbn + \left(1- \cos\left(\frac{\theta}{2}\right)\right) \hbn^2 \right) A_0 \left( I + \sin \left(\frac{\theta}{2}\right) \hbn + \left(1- \cos\left(\frac{\theta}{2}\right)\right) \hbn^2 \right),
\end{equation}
where $\hbu = \theta \hbn$.

Next, we expand the R.H.S. of \eqref{eqn:He-1} and show that it is the same as that of \eqref{B-12-mod}. Indeed, we have
\begin{align}
\label{eqn:He-1-rhs}
\begin{aligned}
 &\left( A_0 - \sin \left(\frac{\theta}{2}\right) \hbn A_0+ \left(1- \cos\left(\frac{\theta}{2}\right)\right) \hbn^2 A_0 \right) \left( I + \sin \left(\frac{\theta}{2}\right) \hbn + \left(1- \cos\left(\frac{\theta}{2}\right)\right) \hbn^2 \right) \\
 &\hspace{1cm}  = A_0 + \sin \frac{\theta}{2} (A_0 \hbn - \hbn A_0) + \left( 1- \cos \frac{\theta}{2} \right)(A_0 \hbn^2 + \hbn^2 A_0) \\
 & \hspace{1.2cm} - \sin^2\left( \frac{\theta}{2} \right) \hbn A_0 \hbn + \left( 1- \cos \frac{\theta}{2} \right)^2 \hbn^2 A_0 \hbn^2,
 \end{aligned}
 \end{align}
 where the terms containing $\hbn A_0 \hbn^2$ and $\hbn^2 A_0 \hbn$ have been cancelled out; this cancellation comes from the fact that $ \hbn A_0 \hbn^2 = \hbn^2 A_0 \hbn$, since $\hbn A_0 \hbn = - (\bn \cdot \az) \hbn$ by $\eqref{eqn:id1}_3$.

By direct calculation, one can show
\begin{align*}
\begin{aligned}
& A_0 \hbn^2 + \hbn^2 A_0 = -A_0 - (\az \cdot \bn) \hbn, \\[2pt]
&\hbn^2 A_0 \hbn^2 = \hbn (-(\bn \cdot \az) \hbn) \hbn = - (\bn \cdot \az) \hbn^3 = (\bn \cdot \az) \hbn.
\end{aligned}
\end{align*}
Using these identities in \eqref{eqn:He-1-rhs}, \eqref{eqn:He-1} becomes:
\begin{align*}
A_1 &= A_0 + \sin \frac{\theta}{2} (A_0 \hbn - \hbn A_0) - \left( 1- \cos \frac{\theta}{2} \right)(A_0 + (\az \cdot \bn) \hbn) \\
&\hspace{0.2cm}  +\sin^2\left( \frac{\theta}{2} \right) (\bn \cdot \az) \hbn  + \left( 1- \cos \frac{\theta}{2} \right)^2 (\bn \cdot \az) \hbn.
\end{align*}
Finally, combine the terms containing $(\az \cdot \bn) \hbn$ to get  \eqref{B-12-mod}.  


While our expression for parallel transport is consistent with this general result on Lie groups, it presents an elementary and 
hands-on approach to its derivation. We are not aware of any other works presenting the explicit forms of the metric tensor, Christoffel symbols, geodesics and parallel transport in exponential coordinates, as done in Sections \ref{sec:2} and \ref{sec:3}. These explicit calculations would be valuable tools for researchers on rigid-body dynamics.



%


\section{The CS model on $\mathrm{SO}(3)$} \label{sec:4}
\setcounter{equation}{0}
In this section, we provide a derivation of the CS model on $\mathrm{SO}(3)$ by gathering calculations from previous sections, and investigate its emergent dynamics. We also study the reduction of the model to $\bbs^1$.
 
\subsection{Formulation of the CS model on $\mathrm{SO}(3)$} \label{sec:4.1}
Let $R_i(t) \in \mathrm{SO}(3)$ be the position of the $i$-th CS particle, and we set its velocity as 
\begin{equation} \label{D-0}
 V_i(t) :=\frac{dR_i(t)}{dt}. 
\end{equation} 
Note that we require $(R_i, V_i)$ to satisfy (see \eqref{A-1}):
\begin{equation} \label{D-0-0}
\frac{DV_i}{dt}  =\frac{\kappa}{N}\sum_{k=1}^N\phi_{ik} (P_{ki} V_k-V_i).
 \end{equation}
Since $V_i(t)$ lies in the tangent space of $\mathrm{SO}(3)$ at $R_i(t)$, there exists a unique skew-symmetric matrix $A_i(t)$ such that
\[
V_i(t)=R_i(t)A_i(t).
\]
This yields
\begin{align}
\begin{aligned} \label{D-0-1}
\frac{d}{dt} V_i(t)&= \frac{d}{dt}(R_i(t)A_i(t))=\left(\frac{d}{dt}R_i(t)\right)A_i(t)+R_i(t)\left(\frac{d}{dt}A_i(t)\right)\\
&=R_i(t)A_i(t)^2+R_i(t)\left(\frac{d}{dt}A_i(t)\right).
\end{aligned}
\end{align}

We will rewrite \eqref{D-0-0} in terms of $A_i$. By definition of the covariant derivative (using the embedding of $\mathrm{SO}(3)$ in $\bbr^{3 \times 3}$), one has:
\begin{equation} \label{D-0-2}
\frac{DV_i(t)}{dt}=\mathrm{Proj}_{T_{R_i(t)} \mathrm{SO}(3)}\left(\frac{d}{dt}V_i(t)\right).
\end{equation}
Then, it follows from \eqref{D-0-1} and \eqref{D-0-2} that 
\begin{align*} \label{D-0-3}
\frac{DV_i(t)}{dt}&=\mathrm{Proj}_{T_{R_i(t)}\mathrm{SO}(3)}\left(R_i(t)A_i(t)^2+R_i(t)\left(\frac{d}{dt}A_i(t)\right)\right) \\
&=R_i(t)\left(\frac{d}{dt}A_i(t)\right),
\end{align*}
where we used the fact that $A_i(t)^2$ is symmetric and $\frac{d}{dt}A_i(t)$ is skew-symmetric. This and \eqref{D-0-0} yield
\begin{equation} \label{D-0-4}
\frac{dA_i(t)}{dt} =\frac{\kappa}{N}\sum_{k=1}^N\phi(R_i, R_k)\left(R_i(t)^\top  P_{ki}V_k(t)-R_i(t)^\top V_i(t)\right).
\end{equation}

We use now the calculations in Section \ref{subsect:pt-coord-free}. Here, $R_k$ and $V_k$ play the roles of $R_0$ and $V(0)$ from these calculations, while $R_i$ and $P_{ki}V_k$ play the roles of $R_1$ and $V(1)$. In particular, $A_k = R_k^\top  V_k$ plays the role of $A_0$, and $R_i^\top  P_{ki}V_k$ plays the role of $A_1$. Also, analogous to $\bu$, $\theta$ and $\bn$ from Section \ref{subsect:pt-coord-free}, we have $ \bu_{ki}, \theta_{ki}$ and $ \bn_{ki}$, defined by:
\begin{align}
\label{eqn:utn-ki}
\begin{aligned} 
& \hbu_{ki}=\log(R_k^\top R_i)=\frac{\theta_{ki}}{2\sin\theta_{ki}}\big(R_k^\top R_i-R_i^\top R_k\big), \\
& \theta_{ki}=\|\mathbf{u}_{ki}\|=\arccos\left(\frac{\mathrm{tr}(R_k^\top R_i)-1}{2}\right), \qquad \hbn_{ki}=\frac{\hbu_{ki}}{\theta_{ki}}.
\end{aligned}
\end{align}

From this framework, apply the parallel transport formula \eqref{B-12-simple} to get
\begin{equation}
\label{eqn:PT-ki}
R_i(t)^\top (P_{ki}V_k(t))=\left(1-\cos\frac{\theta_{ki}}{2}\right)(\mathbf{n}_{ki}\cdot \veeA_k)\widehat{\mathbf{n}}_{ki}+\sin\frac{\theta_{ki}}{2} \, \widehat{\veeA_k\times\mathbf{n}_{ki}}+\cos\frac{\theta_{ki}}{2}A_k,
\end{equation} 
where we denoted 
\begin{equation}
\label{eqn:ak}
\veeA_k = \widecheck{A}_k, \qquad \text { for } k=1,\dots,N. 
\end{equation}
Then, using \eqref{eqn:PT-ki} in \eqref{D-0-4}  yields
\[
\frac{d A_i}{dt}=\frac{\kappa}{N}\sum_{k=1}^N\phi_{ik}\left[ \left(1-\cos\frac{\theta_{ki}}{2}\right)(\mathbf{n}_{ki}\cdot \veeA_k)\widehat{\mathbf{n}}_{ki}+\sin\frac{\theta_{ki}}{2}\widehat{\veeA_k\times \mathbf{n}_{ki}}+\cos\frac{\theta_{ki}}{2}A_k-A_i\right ],
\]
or equivalently, by removing the hats,
\begin{equation} \label{D-0-5}
\frac{d \veeA_i}{dt} =\frac{\kappa}{N}\sum_{k=1}^N\phi_{ik}\left[ \left(1-\cos\frac{\theta_{ki}}{2}\right)(\mathbf{n}_{ki}\cdot \veeA_k)\mathbf{n}_{ki}+\sin\frac{\theta_{ki}}{2}\veeA_k\times \mathbf{n}_{ki}+\cos\frac{\theta_{ki}}{2}\veeA_k-\veeA_i\right].
\end{equation}

Finally, we combine \eqref{D-0} and \eqref{D-0-5} to obtain the CS model  on $\mathrm{SO}(3)$:
\begin{align}
\begin{aligned} \label{D-1}
\frac{dR_i}{dt} &=R_iA_i, \\
\frac{d\veeA_i}{dt} &=\frac{\kappa}{N}\sum_{k=1}^N\phi_{ik}\left[ \left(1-\cos\frac{\theta_{ki}}{2}\right)(\mathbf{n}_{ki}\cdot \veeA_k)\mathbf{n}_{ki}+\sin\frac{\theta_{ki}}{2}\veeA_k\times \mathbf{n}_{ki}+\cos\frac{\theta_{ki}}{2}\veeA_k-\veeA_i\right ],
\end{aligned}
\end{align}
for $t >0$, $1\leq i\leq N$, where we used notations \eqref{eqn:utn-ki} and \eqref{eqn:ak}, and  $\phi_{ik}=\phi(R_i,R_k)$. 

Since parallel transport between two cut locus points cannot be well defined, we should impose some conditions on $\phi$ or a priori conditions (see $({\hyp}_A)$ and   $({\hyp}_B)$).  Some examples of communication weight functions $\phi(R, Q) = \tilde{\phi}(d(R, Q)) $ that satisfy condition $({\hyp}_A)$ are:
\[
\tilde{\phi}(d(R,Q))=\sin(d(R, Q)),\qquad \cos\left(\frac{d(R, Q)}{2}\right),\qquad \cos(d(R, Q))+1.
\]
The well-posedness of the CS model on $\mathrm{SO}(3)$ is given by the following proposition.
\begin{proposition}
Suppose either $({\hyp}_A)$ or $({\hyp}_B)$ holds. Then, the Cauchy problem to system \eqref{D-1}  has a unique global solution.
\end{proposition}
\begin{proof}
\noindent $\bullet$~Case A:~Suppose the first condition $({\hyp}_A)$ holds. Then $\frac{dR_i}{dt}$ and $\frac{d\veeA_i}{dt}$ are Lipschitz continuous functions in the whole domain. Therefore, standard Cauchy-Lipschitz theory and continuity argument based on uniform a priori bound yield the global well-posedness of a smooth solution.  \newline

\noindent $\bullet$~ Case B:~Suppose the second condition $({\hyp}_B)$ holds.  Then  $\frac{dR_i}{dt}$ and $\frac{d\veeA_i}{dt}$ are Lipschitz continuous functions in the whole domain, except at cut locus pairs. Then, similar to Case A, one can guarantee the global unique solvability of the Cauchy problem to \eqref{D-1}.
\end{proof}

\subsection{Reduction from $\mathrm{SO}(3)$ to  $\bbs^1$} \label{sec:4.2} 
In this subsection, we study the relation between the CS model on $\mathrm{SO}(3)$ and the CS model on the circle $\bbs^1$.

Suppose that the initial data are given as follows.
\[
R_i^0=\begin{bmatrix}
\cos\theta_i^0&-\sin\theta_i^0&0\\
\sin\theta_i^0&\cos\theta_i^0&0\\
0&0&1
\end{bmatrix},\quad V_i^0=R_i^0\begin{bmatrix}
0&-\nu_i^0&0\\
\nu_i^0&0&0\\
0&0&0
\end{bmatrix},\quad
A_i^0=\begin{bmatrix}
0&-\nu_i^0&0\\
\nu_i^0&0&0\\
0&0&0
\end{bmatrix}.
\]
Then, by the isomorphism between $\bbr^3$ and $\mathfrak{so}(3)$ in Section \ref{sec:2.1.1}, one has 
\[
\veeA_i^0=(0, 0, \nu_i^0).
\]

Now take the following ansatz for the solution $R_i(t)$ and $\veeA_i(t)$:
\begin{align}\label{D-1-1}
R_i(t)=\begin{bmatrix}
\cos\theta_i(t)&-\sin\theta_i(t)&0\\
\sin\theta_i(t)&\cos\theta_i(t)&0\\
0&0&1
\end{bmatrix},\quad \veeA_i(t)=(0, 0, \nu_i(t)),\quad A_i(t)=\nu_i(t)E_3,
\end{align}
where
\[ \theta_i(0)=\theta_i^0 \qquad \mbox{and} \qquad \nu_i(0)=\nu_i^0. \]
Then we have
\[
\begin{cases}
\displaystyle \theta_{ki}&=\arccos(\cos(\theta_i-\theta_k)),\\
\displaystyle\widehat{\mathbf{u}}_{ki}&=\frac{\theta_{ki}}{2\sin\theta_{ki}}(R_k^\top R_i-R_i^\top R_k)=\frac{\theta_{ki}}{\sin\theta_{ki}}\begin{bmatrix}
0&-\sin(\theta_i-\theta_k)&0\\
\sin(\theta_i-\theta_k)&0&0\\
0&0&0
\end{bmatrix}\\
&\displaystyle=\theta_{ki}\left(\frac{\sin(\theta_i-\theta_k)}{\sin\theta_{ki}}\right) E_3.
\end{cases}
\]
From here, one has
\[
\widehat\bn_{ki}=\left(\frac{\sin(\theta_i-\theta_k)}{\sin\theta_{ki}}\right) E_3\quad\Longrightarrow\quad \bn_{ki}=\left(0, 0,\frac{\sin(\theta_i-\theta_k)}{\sin\theta_{ki}}\right).
\]

By direct calculation, we have
\begin{align*}
\dot{R}_i(t)=\dot{\theta}_i(t)\begin{bmatrix}
-\sin\theta_i(t)&-\cos\theta_i(t)&0\\
\cos\theta_i(t)&-\sin\theta_i(t)&0\\
0&0&0
\end{bmatrix}=\dot{\theta}_i(t)R_i(t) E_3.
\end{align*}
So we set 
\[ \dot{\theta}_i(t)=\nu_i(t) \]
to satisfy $\dot{R}_i=R_iA_i$. Since $\veeA_k$ and $\bn_{ki}$ are parallel, and $\bn_{ki}$ is a unit vector, it holds that 
\[
(\mathbf{n}_{ki}\cdot \veeA_k)\mathbf{n}_{ki}=\veeA_k,\qquad \veeA_k\times \bn_{ki}=0.
\]
From these relations, we have
\[ \frac{\kappa}{N}\sum_{k=1}^N\phi_{ik}\left(\left(1-\cos\frac{\theta_{ki}}{2}\right)(\mathbf{n}_{ki}\cdot \veeA_k)\mathbf{n}_{ki}+\sin\frac{\theta_{ki}}{2}\veeA_k\times \mathbf{n}_{ki}+\cos\frac{\theta_{ki}}{2}\veeA_k-\veeA_i\right)  = \frac{\kappa}{N}\sum_{k=1}^N\phi_{ik}(\veeA_k-\veeA_i). 
\]
Hence, if $(\theta_i, \nu_i)$ satisfy the CS model on the circle, given by:
\begin{align*}
\begin{cases}
\displaystyle\frac{d\theta_i}{dt} =\nu_i,\\
\displaystyle\frac{d\nu_i}{dt} =\frac{\kappa}{N}\sum_{k=1}^N\phi_{ik}(\nu_k-\nu_i),
\end{cases}
\end{align*}
then $(R_i, \veeA_i)$ defined by \eqref{D-1-1} satisfy \eqref{D-1}. This shows that in this case, the CS model on $\mathrm{SO}(3)$ can be reduced to the CS model on the circle.

\subsection{Emergent dynamics} \label{sec:4.3} 
In this subsection, we study the emergent behavior of the CS model \eqref{D-1} by using an explicit Lyapunov functional and LaSalle's invariance principle.  For a given ensemble $\{ R_i, V_i = R_i A_i \}_{i=1}^N$, we define a Lyapunov functional $\mathcal{E}$ which corresponds to the total kinetic energy:
\[
\mathcal{E}(t) :=\sum_{i=1}^N {\| V_i(t)\|}_{R_i(t)}^2.
\]

Then, by direct calculations, one has
\begin{align}
\begin{aligned} \label{D-2}
\frac{d\mathcal{E}(t)}{dt} &=2\sum_{i=1}^N\left\langle \frac{DV_i}{dt}, V_i\right\rangle_{\hspace{-1.5mm}{R_i}} =2\sum_{i=1}^N\left\langle \frac{\kappa}{N}\sum_{k=1}^N\phi(R_i, R_k) (P_{ki}V_k-V_i), V_i \right\rangle_{\hspace{-1.5mm}{R_i}}\\
&=\frac{2\kappa}{N}\sum_{i, k=1}^N \phi(R_i, R_k) {\langle P_{ki}V_k- V_i, V_i\rangle}_{R_i} \\
& =\frac{2\kappa}{N}\sum_{i, k=1}^N\phi(R_i, R_k) \left( {\langle P_{ki} V_k, V_i\rangle}_{R_i}-{\|V_i\|}_{R_i}^2\right)\\
&=\frac{\kappa}{N}\sum_{i, k=1}^N\phi(R_i, R_k)\left(2 {\langle P_{ki}V_k, V_i\rangle}_{R_i}- {\|V_i\|}_{R_i}^2-{\|V_k\|}_{R_k}^2\right)\\
&=-\frac{\kappa}{N}\sum_{i, k=1}^N\phi(R_i, R_k){\|P_{ki}V_k-V_i\|}_{R_i}^2\leq0.
\end{aligned}
\end{align}
For the last two equalities in \eqref{D-2}, we used the symmetry of the communication function, $ \phi(R_i, R_k)=\phi(R_k, R_i)$, and the fact that parallel transport preserves lengths, specifically,
\[
{\|V_k\|}^2_{R_k}={\|P_{ki}V_k\|}^2_{R_i}. 
\]

Let $\{(R_i, A_i)\}_{i=1}^N\in(\mathrm{SO}(3) \times \mathfrak{so}(3))^N$ be a solution of system \eqref{D-1}. From \eqref{D-2}, we know that $\mathcal{E}(t)$ is decreasing, and hence,
\[
\sup_{t\in[0, \infty)}\|V_i(t)\|^2_{R_i(t)}\leq \sup_{t\in[0, \infty)}\mathcal{E}(t)\leq \mathcal{E}(0),
\]
which implies the boundedness of $\frac{1}{2}{\|A_i(t)\|}_{\mathrm{F}}^2={\|V_i(t)\|}_{R_i(t)}^2$ for all $i=1,\dots,N$ and $t\geq 0$. Then, by defining the subset of $\mathfrak{so}(3)$,
\[
\mathcal{D}:=\{A\in\mathfrak{so}(3): {\|A\|}_{\mathrm{F}} \leq \sqrt{2\mathcal{E}(0)}\},
\]
we can restrict the domain of states as follows:
\[
\{(R_i, A_i)\}_{i=1}^N\in(\mathrm{SO}(3) \times \mathcal{D})^N.
\]

Since this domain is compact, we can apply LaSalle's invariance principle to \eqref{D-2} to get
\begin{equation} \label{D-3}
\lim_{t\to\infty}\frac{d\mathcal{E}(t)}{dt} =0, \quad \mbox{i.e.,} \quad \lim_{t\to\infty}\phi(R_i, R_k){\|P_{ki} V_k-V_i\|}_{R_i}^2=0, \qquad\forall~i, k \in \{1, \cdots, N \}.
\end{equation}
On the other hand, using \eqref{B-12-simple} and the fact that the hat operator is an isomorphism between $\mathfrak{so}(3)$ and $\bbr^3$, one has:
\begin{equation}
\label{D-4}
 {\|P_{ki}V_k-V_i\|}_{R_i} =\left\|\left(1-\cos\frac{\theta_{ki}}{2}\right)(\mathbf{n}_{ki}\cdot \veeA_k)\mathbf{n}_{ki}+\sin\frac{\theta_{ki}}{2}\veeA_k\times \mathbf{n}_{ki}+\cos\frac{\theta_{ki}}{2}\veeA_k-\veeA_i\right\|.
\end{equation}

Finally, we combine \eqref{D-3} and \eqref{D-4} to derive the following theorem on emergent dynamics. 
\begin{theorem} \label{T4.1}
Suppose that either $({\hyp}_A)$ or $({\hyp}_B)$ in Section 1 holds, and  let $\{ (R_i, A_i) \}_{i=1}^N$ be a solution to \eqref{D-1}. Then, one has the following assertions.
\begin{enumerate}
\item
Suppose that the condition $({\hyp}_A)$ holds. Then, for each $1\leq i\neq k\leq N$ we have
\begin{equation}
\label{eqn:asymptotic-A}
\lim_{t\to\infty}\phi(R_i, R_k) {\|P_{ki} V_k-V_i\|}_{R_i}^2=0,
\end{equation}
or equivalently, 
\[
\lim_{t\to\infty}\phi(R_i, R_k)\left[ \left(1-\cos\frac{\theta_{ki}}{2}\right)(\mathbf{n}_{ki}\cdot \veeA_k)\mathbf{n}_{ki}+\sin\frac{\theta_{ki}}{2}\veeA_k\times \mathbf{n}_{ki}+\cos\frac{\theta_{ki}}{2}\veeA_k-\veeA_i\right ]= 0.
\]
\item
Suppose that the condition $({\hyp}_B)$ holds. Then, for each $1\leq i\neq k\leq N$ we have
\[
\lim_{t\to\infty} {\|P_{ki} V_k-V_i\|}_{R_i}^2=0,
\] 
or equivalently,
\[
\lim_{t\to\infty}\left[ \left(1-\cos\frac{\theta_{ki}}{2}\right)(\mathbf{n}_{ki}\cdot \veeA_k)\mathbf{n}_{ki}+\sin\frac{\theta_{ki}}{2}\veeA_k\times \mathbf{n}_{ki}+\cos\frac{\theta_{ki}}{2}\veeA_k-\veeA_i\right ]= 0.
\]
\end{enumerate}
\end{theorem}


\subsection{The $\omega$-limit set} \label{sec:4.4} 
In this subsection, we investigate the structure of the $\omega$-limit set of the Cucker-Smale model on $\mathrm{SO}(3)$. By Theorem \ref{T4.1}, each point in the $\omega$-limit set satisfies \eqref{eqn:asymptotic-A}. Moreover, since the $\omega$-limit set is positively invariant, each trajectory that starts in the $\omega$-limit set satisfies $\eqref{D-0-0}$. Therefore, each particle is either at rest or moves along a constant speed geodesic.
\medskip

{\em Geodesic loops in $\mathrm{SO}(3)$.} 
Let us first specify what is a closed geodesic curve on $\mathrm{SO}(3)$.  Consider the unit-speed geodesic curve starting at $R_0$ in the direction $R_0 \hbn$, where $\bn\in\bbr^3$, with $\|\bn\|=1$. By the explicit representation of the exponential map in \eqref{eqn:exp-map}, this curve is given as 
\[
\calR(t) = R_0 \exp(t \hbn) = I + \sin t \, \hbn + (1-\cos t) \hbn^2, \qquad t\in [0,\pi].
\]
At $t=\pi$, $\calR(\pi) = R_0 (2 \bn \bn^\top  - I)$ is a point in the cut locus of $R_0$, at distance $\pi$ from $R_0$. 

Now consider the unit-speed geodesic curve $\calQ(t)$ starting from $R_0$ in the opposite direction $-R_0 \hbn$, as given by
\[
\calQ(t) = R_0 \exp(-t \hbn) = I - \sin t \, \hbn + (1-\cos t) \hbn^2, \qquad t \in [0,\pi].
\]
This curve reaches the same point $\calQ(\pi) = R_0 (2 \bn \bn^\top  - I) = \calR(\pi)$ at time $\pi$.  By construction, 
\[ \calR'(t) = \calR(t) \hbn, \qquad \mbox{and} \qquad \calQ'(t) = - \calQ(t) \hbn, \quad  \mbox{for all $t \in [0,\pi]$}. \]
In particular, the tangents to the two curves at $t=\pi$ are opposite to each other, that is,
\[
\calR'(\pi) = \calR(\pi) \hbn =  \calQ(\pi) \hbn = -\calQ'(\pi).
\]

To obtain a geodesic loop starting and ending at $R_0$ we patch together the two geodesic arcs, and consider the curve $\tilde{\calR}(t)$ on $[0,2\pi]$ defined by:
\begin{equation*}
\tilde{\calR}(t) 
 = \begin{cases}
          \calR(t),          & \text{for } t \in [0,\pi], \\
          \calQ(2 \pi -t), & \text{for } t \in (\pi,2\pi].
      \end{cases}
\end{equation*} 
Then, we can extend the curve for all $t >2 \pi$, with the image of the curve being this closed geodesic loop. 

In the visualization of $\mathrm{SO}(3)$ as a ball of radius $\pi$ with center at $R_0$ (see also Section \ref{sec:5}), this geodesic loop consists of two parts. The first part is the curve $\calR(t)$ starting at the center and reaching the boundary (the sphere of radius $\pi$) at $t=\pi$. The second part of the loop starts at the antipodal of the point where the first curve ended at $t=\pi$, and ends at the center $R_0$. So visually, it seems like  that a jump has occurred at $t=\pi$, however, antipodal points on the boundary are identified, and the curve is indeed smooth.
\medskip

{\em Structure of the $\omega$-limit set.} Consider $N$ constant speed geodesic loops $R_i(t)$ on $\mathrm{SO}(3)$, as described above, initialized in the $\omega$-limit set of \eqref{D-1}. The loops start at $R_i^0 \in \mathrm{SO}(3)$, in the direction $R_i^0 A_i$, where $A_i = \widehat{\veeA}_i \in \mathfrak{so}(3)$, $i=1,\dots,N$, are constant matrices. Note that $\veeA_i$ have not been assumed of unit length, so a generic loop $R_i(t)$ closes after $t=2\pi/\|\veeA_i\|$ units of time.  For all $i$, we have 
\[
\dot{R}_i(t) = R_i(t) A_i.
\]

For the moment, we ignore  the issue of points reaching the cut locus of each other dynamically.  To characterize the $\omega$-limit set of the CS model \eqref{D-1}, we require 
\begin{equation}
\label{eqn:PT}
P_{ki} \dot{R}_k(t) = \dot{R}_i(t), \qquad\text{for all } i, k\in\{1, 2, \cdots, N\} \quad \text{ and } t\geq0,
\end{equation}
such that $\theta_{ki}(t) = d(R_i(t),R_k(t))<\pi$. The goal is to show that the only possibility for this to happen is to have a single geodesic loop along which all particles move. 

Fix two arbitrary indices $i$ and $k$ and suppose that $\theta_{ki}(0)\neq \pi$. Then there exists $\epsilon>0$ such that
\[
\theta_{ki}(t) \neq \pi,\quad \forall~t\in[0, \epsilon).
\] 
Consider $\gamma_{ki}(t,\cdot)$, the length-minimizing geodesic between $R_k(t)$ and $R_i(t)$, and the parallel transport of $\dot{R}_k(t)$ along $\gamma_{ki}$, from $R_k(t)$ to $R_i(t)$. Using the calculations detailed in Section \ref{sec:4.1} (see equation \eqref{eqn:PT-ki}), the parallel transport relationship \eqref{eqn:PT} can be written as:
\begin{equation}
\label{PT-ij-simple}
A_i=\left(1-\cos\frac{\theta_{ki}}{2}\right)  (\bn_{ki} \cdot \veeA_k)\, \hbn_{ki}
+\sin\frac{\theta_{ki}}{2}\,  \widehat{\veeA_k \times \bn_{ki}} +\cos\frac{\theta_{ki}}{2} A_k,
\end{equation}
or without the hats,
\begin{equation}
\label{PT-ij-geom}
\veeA_i =\left(1-\cos\frac{\theta_{ki}}{2}\right)  (\bn_{ki} \cdot  \veeA_k)\, \bn_{ki}
+\sin\frac{\theta_{ki}}{2}\,   \veeA_k\times \bn_{ki} +\cos\frac{\theta_{ki}}{2} \, \veeA_k,
\end{equation}
where we used notations \eqref{eqn:utn-ki} and \eqref{eqn:ak}. Note that $\theta_{ki}$ and $\bn_{ki}$ depend on time, while $\ba_i$ and $\ba_k$ are fixed. Equations \eqref{PT-ij-simple} (and \eqref{PT-ij-geom}) need to hold for all times.

First, since parallel transport preserves lengths, we have
\[  {\|R_i A_i\|}_{R_i} = {\| R_k A_k\|}_{R_k}, \]
or equivalently, $\| \veeA_i\| = \| \veeA_k\|$. We denote by $\theta$ the common length, and hence we can write
\begin{equation}
\label{eqn:aiak}
\veeA_i= \theta \bn_i \quad \text{and}\quad \veeA_k= \theta \bn_k,
\end{equation}
where $\bn_i$ and $\bn_k$ are unit vectors in $\bbr^3$.

Second, parallel transport preserves angles, so the angle that $\dot{R}_k(t)$ makes with the geodesic $\gamma_{ki}$ at $R_k(t)$ is the same as the angle between $\dot{R}_i(t)$ and $\gamma_{ki}$ at $R_i(t)$. Consequently,
\[
\dot{R}_k \cdot \log_{R_k} R_i = - \dot{R}_i \cdot \log_{R_i} R_k,
\]
and hence, we have
\begin{equation} \label{New-2}
R_k A_k \cdot R_k \log(R_k^\top  R_i) = - R_i A_i \cdot R_i \log(R_i^\top  R_k).
\end{equation}
By the expression of matrix logarithm (see \eqref{eqn:Rtotv}), we can write \eqref{New-2} as
\[
\frac{\theta_{ki}}{2 \sin \theta_{ki}} {\langle A_k , R_k^\top  R_i - R_i^\top  R_k \rangle}_{\mathrm{F}} = - \frac{\theta_{ki}}{2 \sin \theta_{ki}} {\langle A_i, R_i^\top  R_k - R_k^\top  R_i \rangle}_{\mathrm{F}},
\]
and furthermore, by restoring the dependence on $t$, 
\[
\frac{\theta_{ki}(t)}{2 \sin \theta_{ki}(t)} {\langle A_k - A_i , R_k(t)^\top  R_i(t) - R_i(t)^\top  R_k(t)\rangle}_{\mathrm{F}} =0, \qquad \text{ for all } t\geq 0.
\]
The expression $\frac{\theta_{ki}}{2 \sin \theta_{ki}}$ is well-defined when $\theta_{ki}\in(0, \pi)$. Since the function can be extended to $\theta_{ki}\in[0, \pi)$ continuously, we define the value at $\theta_{ki}=0$ as $\frac{1}{2}$.

Now we work out the following equation explicitly:
\begin{equation}
\label{eqn:orthog}
\langle A_k - A_i , R_k(t)^\top  R_i(t) - R_i(t)^\top  R_k(t) \rangle_{\mathrm{F}} =0.
\end{equation}
Note that for any skew-symmetric matrix $A$ and rotation matrix $R$, one has
\[
\langle A,R^\top  \rangle_{\mathrm{F}} = \Tr(A R) = \Tr(R A) = - \Tr(RA^\top ) = - \langle A,R \rangle_{\mathrm{F}}
\]
Hence, since $A_k - A_i$ is skew-symmetric and $R_k(t)^\top  R_i(t) \in \mathrm{SO}(3)$, \eqref{eqn:orthog} can be written as
\begin{equation*}
{\langle A_k - A_i , R_k(t)^\top  R_i(t) \rangle}_{\mathrm{F}} =0,
\end{equation*}
or equivalently,
\begin{equation}
\label{eqn:orthog1}
\langle  A_i , R_k(t)^\top  R_i(t) \rangle_{\mathrm{F}} = \langle  A_k , R_k(t)^\top  R_i(t) \rangle_{\mathrm{F}}.
\end{equation}
Since $A_i =\widehat{\veeA}_i = \theta \hbn_i$ and $A_k = \widehat{\veeA}_k = \theta \hbn_k$ (see \eqref{eqn:aiak}), we can cancel $\theta$ and write \eqref{eqn:orthog1} as:
\begin{equation}
\label{eqn:orthog2}
\langle  \hbn_i, R_k(t)^\top  R_i(t) \rangle_{\mathrm{F}} = \langle  \hbn_k , R_k(t)^\top  R_i(t) \rangle_{\mathrm{F}}.
\end{equation}

In the following lemma, we calculate ${\langle  \hbn_i , R_k(t)^\top  R_i(t) \rangle}_{\mathrm{F}}$ and ${\langle  \hbn_k , R_k(t)^\top  R_i(t) \rangle}_{\mathrm{F}}$.
\begin{lemma}\label{L4.1}
The following relations hold:
\begin{align*}
\begin{aligned}
& (i)~\langle  \hbn_i , R_k(t)^\top  R_i(t) \rangle_{\mathrm{F}}  =  \Tr \left( -\frac{1}{2} \hbn_i \Rz \hbn_k^2- \frac{1}{2} \hbn_i^2 \Rz \hbn_k \right) \\[2pt]
& \hspace{2cm} + \Tr \left(- \hbn_i^2 \Rz - \hbn_i^2 \Rz \hbn_k^2 \right) \sin(\theta t) + \Tr \left(\hbn_i \Rz + \hbn_i \Rz \hbn_k^2\right) \cos(\theta t)  \\[2pt]
& \hspace{2cm}+ \Tr \left(\frac{1}{2} \hbn_i \Rz \hbn_k + \frac{1}{2} \hbn_i^2 \Rz \hbn_k^2 \right) \sin(2 \theta t) + \Tr \left(-\frac{1}{2} \hbn_i \Rz \hbn_k^2+ \frac{1}{2} \hbn_i^2 \Rz \hbn_k \right) \cos(2 \theta t), \\[5pt]
& (ii)~\langle  \hbn_k , R_k(t)^\top  R_i(t) \rangle_{\mathrm{F}} = \Tr \left( -\frac{1}{2} \hbn_i \Rz \hbn_k^2 -\frac{1}{2} \hbn_i^2 \Rz \hbn_k \right) \\[2pt]
& \hspace{2cm} + \Tr \left(\Rz \hbn_k^2 +  \hbn_i^2 \Rz \hbn_k^2 \right) \sin(\theta t) + \Tr \left( \Rz \hbn_k + \hbn_i^2 \Rz \hbn_k \right) \cos(\theta t) \\[2pt]
& \hspace{2cm}+ \Tr \left(-\frac{1}{2} \hbn_i \Rz \hbn_k -\frac{1}{2} \hbn_i^2 \Rz \hbn_k^2 \right) \sin(2 \theta t) + \Tr \left(\frac{1}{2} \hbn_i \Rz \hbn_k^2 - \frac{1}{2} \hbn_i^2 \Rz \hbn_k \right) \cos(2 \theta t),
\end{aligned}
\end{align*}
where we used the notation $\Rz = {R_i^0}^\top  R_k^0$.
\end{lemma}
\begin{proof}
Since the proof of this lemma is rather lengthy, we present it in Appendix \ref{proofL4.1}.
\end{proof}

We substitute the result of Lemma \ref{L4.1} into \eqref{eqn:orthog2} and equate the coefficients of the trigonometric functions $\sin(\theta t)$, $\cos(\theta t)$, $\sin(2\theta t)$, and $\cos(2\theta t)$ to obtain
\begin{equation} \label{J-1}
\begin{cases}
\displaystyle  \Tr(- \hbn_i^2 \Rz - \hbn_i^2 \Rz \hbn_k^2)  = \Tr(\Rz \hbn_k^2 +  \hbn_i^2 \Rz \hbn_k^2),\vspace{0.2cm} \\
\displaystyle  \Tr(\hbn_i \Rz + \hbn_i \Rz \hbn_k^2)  = \Tr( \Rz \hbn_k + \hbn_i^2 \Rz \hbn_k),\vspace{0.2cm} \\
\displaystyle \Tr \left(\frac{1}{2} \hbn_i \Rz \hbn_k + \frac{1}{2} \hbn_i^2 \Rz \hbn_k^2\right) = \Tr\left(-\frac{1}{2} \hbn_i \Rz \hbn_k -\frac{1}{2} \hbn_i^2 \Rz \hbn_k^2\right),\vspace{0.2cm}\\
\displaystyle \Tr\left(-\frac{1}{2} \hbn_i \Rz \hbn_k^2+ \frac{1}{2} \hbn_i^2 \Rz \hbn_k \right)  =  \Tr\left(\frac{1}{2} \hbn_i \Rz \hbn_k^2 - \frac{1}{2} \hbn_i^2 \Rz \hbn_k\right).
\end{cases}
\end{equation}
From $\eqref{J-1}_4$, we get
\begin{equation}
\label{rel1}
\Tr( \hbn_i \Rz \hbn_k^2) = \Tr( \hbn_i^2 \Rz \hbn_k).
\end{equation}
Then, \eqref{rel1} and   $\eqref{J-1}_2$ yield
\begin{equation}
\label{rel2}
\Tr(\Rz \hbn_i ) = \Tr( \Rz \hbn_k).
\end{equation}
Note that $\eqref{J-1}_3$ can be rewritten as
\begin{equation}
\label{rel3}
 \Tr(\hbn_i \Rz \hbn_k + \hbn_i^2 \Rz \hbn_k^2) = 0.
\end{equation}
Finally,  $\eqref{J-1}_1$ leads to
\begin{equation}
\label{rel4}
\Tr(\Rz \hbn_k^2 + \Rz \hbn_i^2 + 2  \hbn_i^2 \Rz \hbn_k^2) = 0.
\end{equation}

Without any loss of generality, we assume that $\Rz$ is a rotation about the $x_1$-axis. Since $\Rz = {R_i^0}^\top  R_k^0$, the angle of rotation is $\theta_{ki}^0$. For simplicity of notation, set 
\[ \alpha = \theta_{ki}^0. \]
Then $\Rz$ becomes 
\begin{align*}
\label{Rz}
\Rz=\begin{bmatrix}
1 &0 &0\\
0& \cos \alpha&-\sin \alpha\\
0& \sin \alpha& \cos \alpha
\end{bmatrix}.
\end{align*}
Also for ease of notation, we set 
\[ \bn_i = (x_1,x_2,x_3), \qquad \mbox{and} \qquad \bn_k = (y_1,y_2,y_3). \]

By direct calculation, one finds
\[
\Tr(\Rz (\hbn_i - \hbn_k)) = 2 \sin \alpha (y_1 - x_1).
\]
For $\alpha < \pi$, \eqref{rel2} implies 
\[ y_1=x_1. \]
By direct calculations  one finds
\begin{align*}
\begin{aligned}
& \Tr(\hbn_i \Rz \hbn_k) = -((\cos \alpha +1)x_2+\sin \alpha \, x_3) y_2 + (\sin \alpha \, x_2 - (\cos \alpha +1) x_3) y_3 - 2 x_1 y_1 \cos \alpha, \\[3pt]
& \Tr(\Rz \hbn_i^2) = (1-\cos \alpha) x_1^2 -1 - \cos \alpha.
\end{aligned}
\end{align*}
Note that since $\Tr(\Rz \hbn_i^2)$ depends only on $x_1$, one has
\[ \Tr(\Rz \hbn_i^2) = \Tr(\Rz \hbn_k^2). \]
Hence, we can combine \eqref{rel3} and \eqref{rel4} to get
\[
\Tr( \Rz \hbn_i^2) = \Tr(\hbn_i \Rz \hbn_k),
\]
and by using the coordinate expressions for the two traces, we find
\begin{equation}
\label{eqn:eqn1}
 -((\cos \alpha +1)x_2+\sin \alpha \, x_3) y_2 + (\sin \alpha \, x_2 - (\cos \alpha +1) x_3) y_3 = (\cos \alpha +1)(x_1^2 -1).
\end{equation}

By using $y_1=x_1$ one finds that equation \eqref{rel1} leads to an identity, with both sides being equal to
\begin{align*}
& x_1(\sin \alpha \, x_2 + (1-\cos \alpha) x_3) y_2 + x_1((\cos \alpha -1)x_2 + \sin \alpha \, x_3) y_3  \\
& \hspace{5cm} + x_1(1+x_1^2) \sin \alpha + ((x_3^2-1) \cos \alpha - x_2 x_3 \sin \alpha.
\end{align*}
Finally, \eqref{rel3} yields
\begin{equation}
\label{eqn:eqn2}
\begin{aligned}
& ((x_2^2-1) \cos \alpha + x_2 x_3 \sin \alpha) y_2^2 + ((x_3^2-1) \cos \alpha - x_2 x_3 \sin \alpha) y_3^2 \\
& \qquad \qquad + (-\sin \alpha \, x_2^2 + \sin \alpha \, x_3^2 + 2 \cos \alpha \, x_2 x_3) y_2 y_3 = x_1^2 (x_1^2-1) \cos \alpha.
\end{aligned}
\end{equation}

We solve equations \eqref{eqn:eqn1} and \eqref{eqn:eqn2} for unknowns $y_2$ and $y_3$. By direct (but tedious) calculations, the two sets of solutions turn out to be:
\[
\begin{cases}
\displaystyle y_2 = x_2, \quad y_3 = x_3; \quad \mbox{and} \\
\displaystyle y_2 = \cos \alpha \, x_2 + \sin \alpha \, x_3, \quad y_3 = - \sin \alpha x_2 + \cos \alpha x_3,
\end{cases}
\]
or in vector form,
\[
\begin{cases}
\bn_k = \bn_i; \quad \mbox{and} \\
\bn_k = {\Rz}^\top  \bn_i.
\end{cases}
\]
Note that the latter solution satisfies
\[
R_k^0 \bn_k = R_i^0 \bn_i.
\]
Next we study the two cases separately. \newline

\noindent $\bullet$~Case A~ ($\bn_i = \bn_k:=\bn$; also, $\veeA_i = \veeA_k = \theta \bn$, and $A_i= A_k = \theta \hbn:= A$): Recall from the geometric interpretation of \eqref{PT-ij-geom} that at all times, $\veeA_i$ can be obtained as  the rotation by right-hand-rule of $\veeA_k$ about the axis $-\bn_{ki}$, by an angle $\frac{\theta_{ki}}{2}$. In particular, this statement holds at $t=0$. Provided $\theta_{ki}^0:= \theta_{ki}(0) \neq 0$, a rotation keeps a vector unchanged only if the vector is parallel to the rotation axis. Hence, $\bn \parallel \bn_{ki}^0$, which implies that the particles move along the same geodesic loop, which includes the geodesic curve between $R_i^0$ and $R_k^0$. In the trivial case $\theta_{ki}^0 = 0$, one gets $R_i^0=R_k^0$ and $R_i(t) = R_k(t)$ for all $t$. Note that moving along the same geodesic loop, the distance between the two particles remains constant in time, that is, $\theta_{ki}(t) = \theta_{ki}^0$ for all $t$. 

\medskip
\noindent $\bullet$ Case B~ ($\bn_i =  {R_i^0}^\top  R_k^0 \, \bn_k $, or equivalently $\veeA_i =  {R_i^0}^\top  R_k^0 \veeA_k $): Note that ${R_i^0}^\top  R_k^0$ is a rotation about $-\bn_{ki}^0$ by angle $\theta_{ki}^0$. This can only be consistent with the geometric interpretation of \eqref{PT-ij-geom} provided $\theta_{ki}^0 =0$ or equivalently $R_i^0=R_k^0$, in which case we get $R_i(t) = R_k(t)$ for all $t$.

The considerations above lead to the following proposition.
\begin{proposition}\label{P4.3}
Let $\{(R_i, A_i)\}_{i=1}^N$ be a solution of \eqref{D-1} which lies in the $\omega$-limit set. Then there exists a constant matrix $A = \widehat{\veeA}\in\mathfrak{so}(3)$ such that 
\begin{align*}
A_i=A,\qquad \text{ for all }i=1, 2, \dots, N.
\end{align*}
Furthermore, if $\theta_{ki}^0<\pi$ for all $i,k$, then $\bu_{ki}$ is parallel to $\veeA$ (see notations \eqref{eqn:utn-ki}), and this implies that all particles rotate along the same closed geodesic curve.
\end{proposition}
\begin{proof}
Fix an index $i \in \{1,\dots,N\}$ and take another generic index $k \neq i$. If $\theta_{ki}^0<\pi$, the discussion above applies and we have 
\begin{align*}
A_i=A_k :=A ,\qquad  \bu_{ki} \parallel \veeA,\qquad\text{and}\quad \theta_{ki}(t)=\theta_{ki}^0<\pi\quad\forall t\geq0,
\end{align*}
where $\widehat{\bu}_{ki}=\log(R_k^\top R_i)$ and $\hba = A$. Furthermore, $R_i$ and $R_k$ rotate on the same closed geodesic curve. Note that when $\theta_{ki}^0=0$ this reduces to the trivial case $R_i(t) = R_k(t)$ and $\hbu_{ki} = 0$, for all $t$.

Consider now the case $\theta_{ki}^0=\pi$. The matrices $R_i$ and $R_k$ are governed by the dynamics
\[
\dot{R}_i=R_i A_i,\quad \dot{R}_k=R_k A_k,
\]
for constant matrices $A_i, A_k\in\mathfrak{so}(3)$. If $A_i\neq A_k$, then $0<\theta_{ki}(t)<\pi$ for $0<t<\epsilon$, where $\epsilon$ is a sufficiently small positive constant. Then we can restart the time and apply the result for the first case, to obtain that $A_i=A_k$, which gives a contradiction. Therefore, we have $A_i=A_k$ in this case as well.

When particles satisfy $\theta_{ki}^0<\pi$ for all $i,k$, the first case applies to the entire group and all particles rotate along the same closed geodesic curve.
\end{proof}

The main result of this section is the following dichotomy on the long time behavior of the CS model on $\textrm{SO(3)}$. 
\begin{theorem}\label{T4.2}
Let $\{(R_i(t), A_i(t))\}_{i=1}^N$ be a solution of \eqref{D-1}. Then, we have the following dichotomy for its asymptotic dynamics:
\begin{enumerate}
\item either the kinetic energy tends to zero:
\[
\lim_{t\to\infty}\mathcal{E}(t)=0,
\]
\item or the kinetic energy converges to a nonzero positive value  $\mathcal{E}^\infty$, and 
\begin{equation}
\label{VAM}
\lim_{t\to\infty}(A_i(t)-A_k(t))=0,\qquad  \text{ for all }1\leq i, k\leq N,
\end{equation}
\begin{equation}
\label{eqn:lim-norm}
\lim_{t \to \infty} {\|A_i(t)\|}_\mathrm{F} = \frac{2}{N}\mathcal{E}^\infty, \qquad \text{ for all } 1 \leq i \leq N.
\end{equation}
\end{enumerate}
\end{theorem}
\begin{proof}
It follows from \eqref{D-2} that $\mathcal{E}(t)$ is non-increasing. Since $\mathcal{E}(t)\geq0$, there exists a limit:
\begin{align}\label{EC}
\lim_{t\to\infty}\mathcal{E}(t)=\mathcal{E}^\infty.
\end{align}
If $\mathcal{E}^\infty=0$, we can obtain the first case. Now we assume that $\mathcal{E}^\infty>0$.

By LaSalle's invariance principle we infer that
\begin{align}\label{Las}
\lim_{t\to\infty}\mathrm{dist}(\mathcal{X}(t), \mathcal{S})=0,
\end{align}
where $\mathcal{X}(t)=\{(R_i(t), A_i(t))\}_{i=1}^N$ and $\mathcal{S}$ is the $\omega$-limit set of system \eqref{D-1}. From Proposition \ref{P4.3}, we know that any state
$\{(R_i^\ast, A_i^\ast)\}_{i=1}^N$ in $\mathcal{S}$ satisfies 
\begin{align*}
{A}_i^\ast \equiv A^\ast \in\mathfrak{so}(3), \qquad \textrm{ for all } i=1,2,\dots.N.
\end{align*}
If we combine this fact with \eqref{Las}, we can then obtain \eqref{VAM}.

Note that the asymptotic behavior in \eqref{VAM} only states that the differences between $A_i(t)$ and $A_k(t)$ converge to zero for all $i,k$. In particular, \eqref{VAM} does not imply that each $A_i(t)$ converges to some fixed $A\in \mathfrak{so}(3)$ as $t \to \infty$. On the other hand, one can show that the norms ${\|A_i(t)\|}_\mathrm{F}$ have a common fixed limit as $t \to \infty$.
To show this, express the energy functional $\mathcal{E}$ as:
\[
\mathcal{E}(t)=\sum_{k=1}^N {\|V_k(t)\|}_{R_k(t)}^2=\frac{1}{2}\sum_{k=1}^N {\|A_k(t)\|}_\mathrm{F}^2.
\]
By \eqref{EC} we then have
\begin{align}\label{ZZ-1}
\lim_{t\to\infty}\sum_{k=1}^N {\|A_k(t)\|}_\mathrm{F}^2=2\mathcal{E}^\infty.
\end{align}

For a fixed index $1\leq i\leq N$, we have 
\begin{align*}
\sum_{k=1}^N {\|A_k\|}_\mathrm{F}^2&=\sum_{k=1}^N {\|A_i+(A_k-A_i)\|}_\mathrm{F}^2\\
&=\sum_{k=1}^N\left({\|A_i\|}_\mathrm{F}^2+2 \langle A_i, A_k-A_i\rangle_\mathrm{F}+ {\|A_i-A_k\|}_\mathrm{F}^2\right)\\
&=N {\|A_i\|}_\mathrm{F}^2+2\sum_{k=1}^N\langle A_i, A_k-A_i\rangle_\mathrm{F}+\sum_{k=1}^N {\|A_i-A_k\|}_\mathrm{F}^2.
\end{align*}
Letting $t\to\infty$ in the equation above and using \eqref{ZZ-1}, we get
\begin{align}\label{ZZ-2}
2\mathcal{E}^\infty=\lim_{t\to\infty}\left( N {\|A_i\|}_\mathrm{F}^2+2\sum_{k=1}^N\langle A_i, A_k-A_i\rangle_\mathrm{F}+\sum_{k=1}^N {\|A_i-A_k\|}_\mathrm{F}^2 \right).
\end{align}
Now combine \eqref{VAM} and \eqref{ZZ-2} to obtain
\[
2\mathcal{E}^\infty=N\lim_{t\to\infty}{\|A_i(t)\|}_\mathrm{F},
\]
which yields \eqref{eqn:lim-norm}.


\end{proof}
\begin{remark}
As noted above, \eqref{VAM} does not imply that each $A_i(t)$ converges  to a fixed skew-symmetric matrix as $t \to \infty$. Consequently, the directions of motion of the particles are not guaranteed to have a limit as $t \to \infty$. By \eqref{eqn:lim-norm} however, the particles' speeds $\|V_i(t)\|_{R_i(t)}=\frac{1}{2}{\|A_i(t)\|}_\mathrm{F}$ become equal asymptotically. Numerical simulations suggest (see Section \ref{sec:5}) that the directions of motion of the particles have in fact a common limit as $t \to \infty$, meaning that asymptotically, particles approach and rotate with constant speed along a common geodesic.
\end{remark}

\section{Numerical Simulations} \label{sec:5}
\setcounter{equation}{0}
We solve numerically the CS model on $\mathrm{SO}(3)$ using the 4th order Runge-Kutta method. In the simulations presented below we initialize the rotation matrices $R_i$ in the angle-axis representation (see Section \ref{sec:2.1.2}). Specifically, the rotation angles $\theta_i$ ($i=1,\dots,N$) were initialized randomly in the interval $(0,\pi/2)$, while the unit vectors $\mathbf{v}_i$ were generated in spherical coordinates, with the polar and azimuthal angles drawn randomly in the intervals $(0,\pi)$ and $(0,2 \pi)$, respectively. For the initial velocities $R_i A_i$ ($i=1,\dots,N$), we generate the components of $\veeA_i$ (here, $\widehat{\veeA_i} = A_i$) randomly in the interval $(-1,1)$.

For plotting purposes, $\mathrm{SO}(3)$ is identified with the ball in $\bbr^3$ of radius $\pi$, centered at the origin. The center of the ball corresponds to the identity matrix $I$. A generic point within the ball represents a rotation matrix, with rotation angle given by the distance from the point to the center, and axis given by the ray from the center to the point. Antipodal points on the surface of the ball are identified, as they represent the same rotation matrix (rotation by $\pi$ about a ray gives the same result as rotation by $\pi$ about the opposite ray).

The numerical results we present correspond to the communication function
\[
\phi(R,Q) = \cos\left( \frac{d(R,Q)}{2}\right).
\]
Note that this function vanishes for pairs of cut points, i.e., for $d(R,Q)=\pi$. Similar results were obtained with other communication functions as well.

Figure \ref{fig:stationary} corresponds to a simulation where all particles come to a stop asymptotically. In Figure \ref{fig:stationary}(a) the initial and final locations of the particles are indicated by black dots and red diamonds, respectively. At $t=10,000$ particles have reached a configuration with an energy $\mathcal{E}$ of order $10^{-9}$. Also note the decay to $0$ over time of the energy, as shown in Figure \ref{fig:stationary}(b).
\begin{figure}[thb]
 \begin{center}
 \begin{tabular}{cc}
 \includegraphics[width=0.55\textwidth]{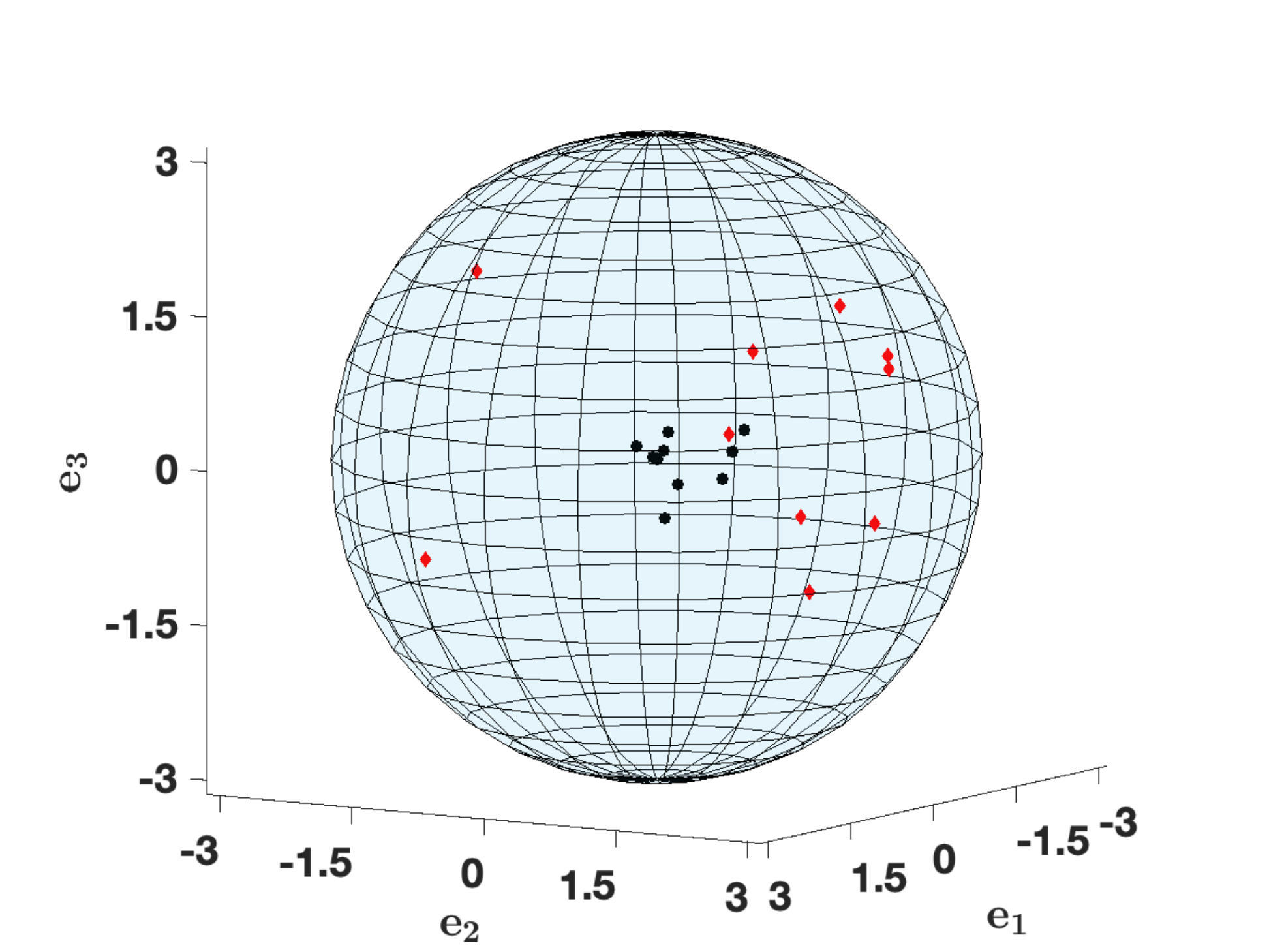}  &
 \includegraphics[width=0.45\textwidth]{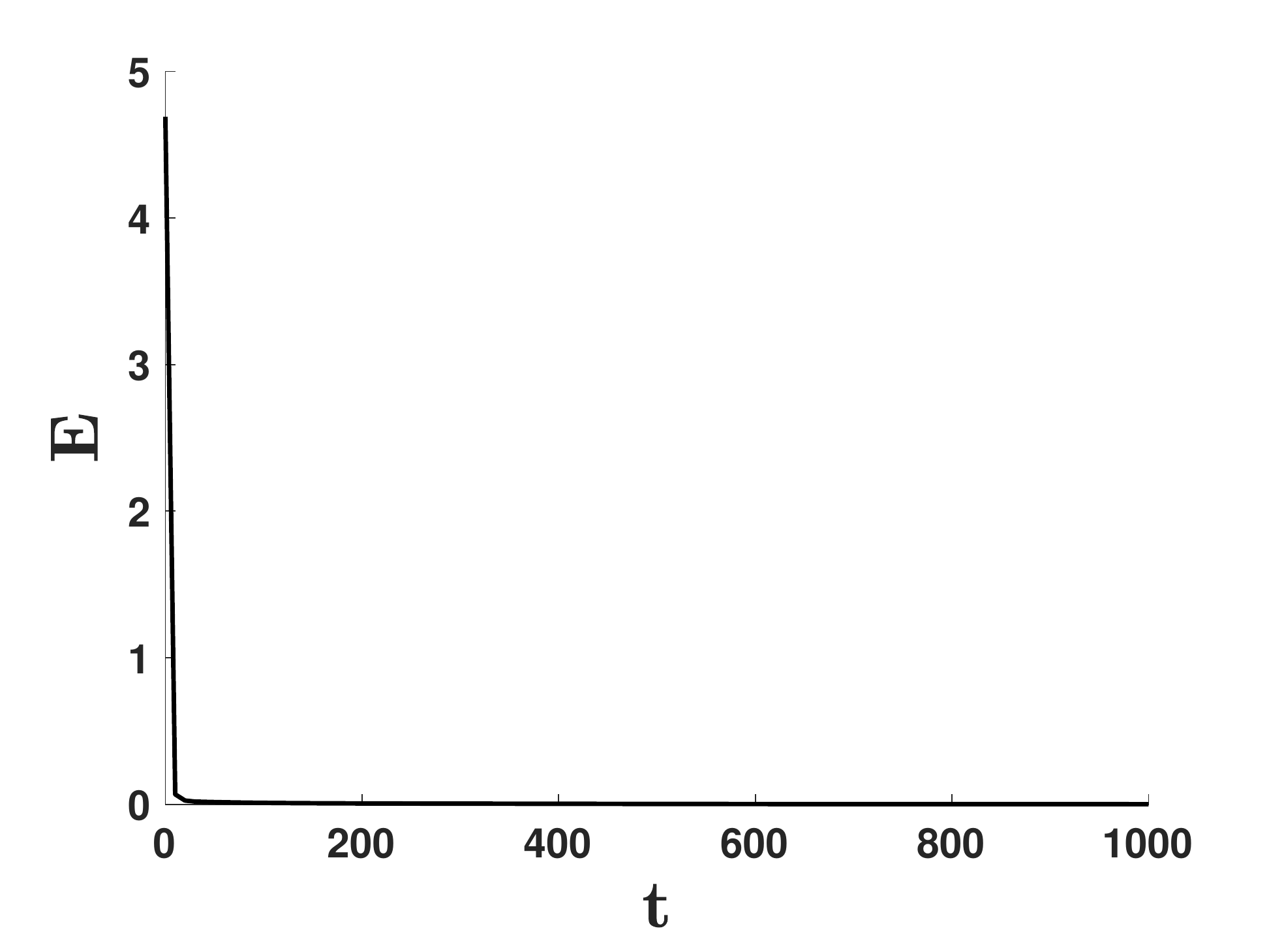}  \\
 (a) & (b) 
 \end{tabular}
 \begin{center}
 \end{center}
\caption{A group of $N=10$ particles reach a stationary state where all particles are at rest. (a) The initial state is indicated by black dots. Red diamonds indicate the particle locations at $t=10000$, where the particle configuration has reached an energy of order $10^{-9}$.  (b) Energy decaying to $0$ over time.}
\label{fig:stationary}
\end{center}
\end{figure}

Figure \ref{fig:flocking} shows a simulation in which particles reach asymptotically a flocking state. In Figure \ref{fig:flocking}(a) the initial locations of the particles are indicated by black dots. At $t=2,000$ the particles (red diamonds) have aligned along the closed geodesic path indicated by small blue dots, and we will rotate on this closed geodesic loop indefinitely. Note that the initial and end points of the geodesic path are antipodal of each other and hence, by the visualization of $\mathrm{SO}(3)$ that we use, they represent the same rotation matrix. Therefore, the geodesic path is in fact closed, though it does not appear so in the figure.  Figure \ref{fig:flocking}(b) shows the energy decay over time. As opposed to the previous simulation, the energy now approaches a non-zero value as $t\to \infty$, as particles do not stop, but keep moving on the geodesic loop.

\begin{figure}[thb]
 \begin{center}
 \begin{tabular}{cc}
 \includegraphics[width=0.58\textwidth]{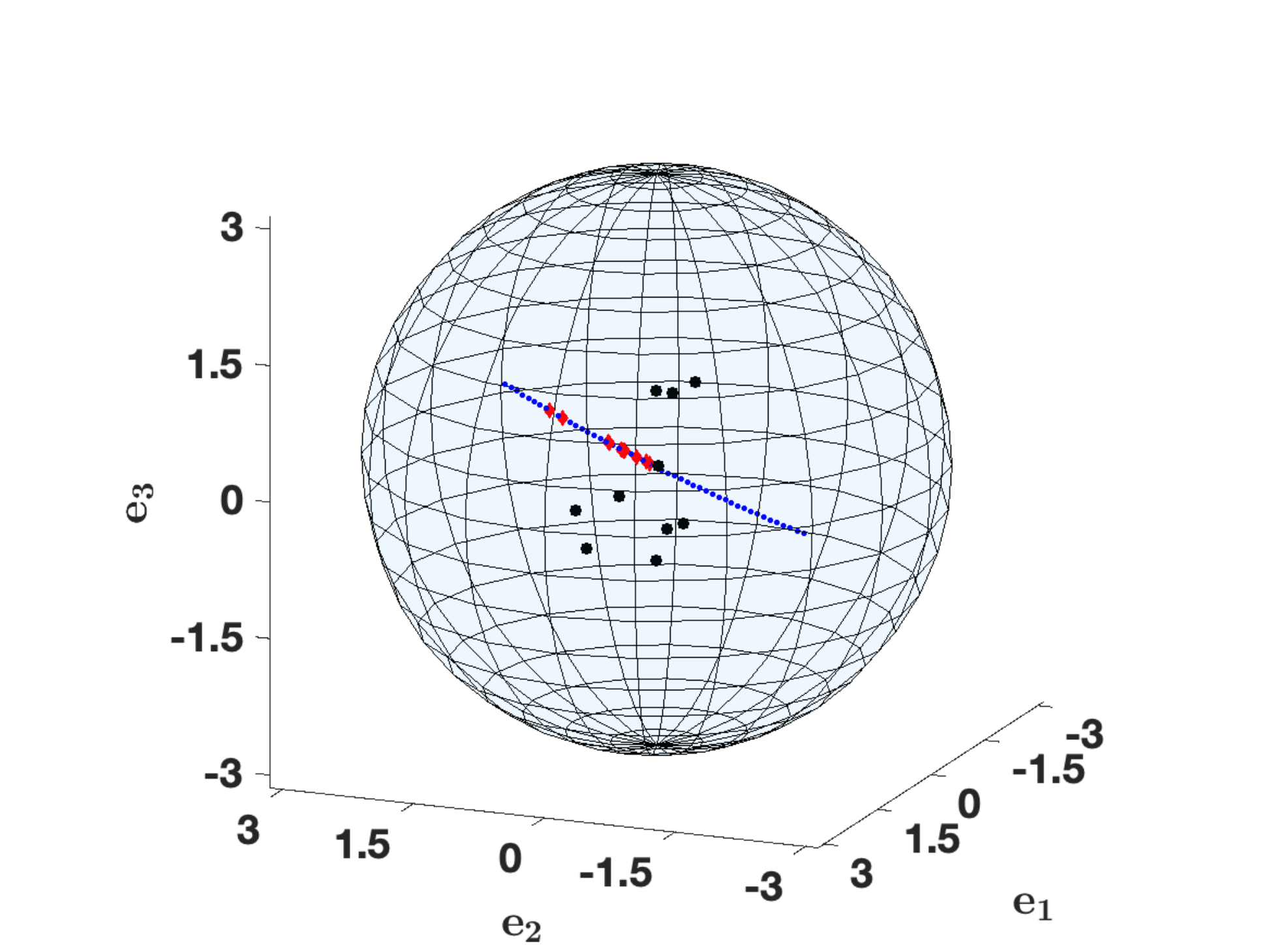} & 
 \includegraphics[width=0.42\textwidth]{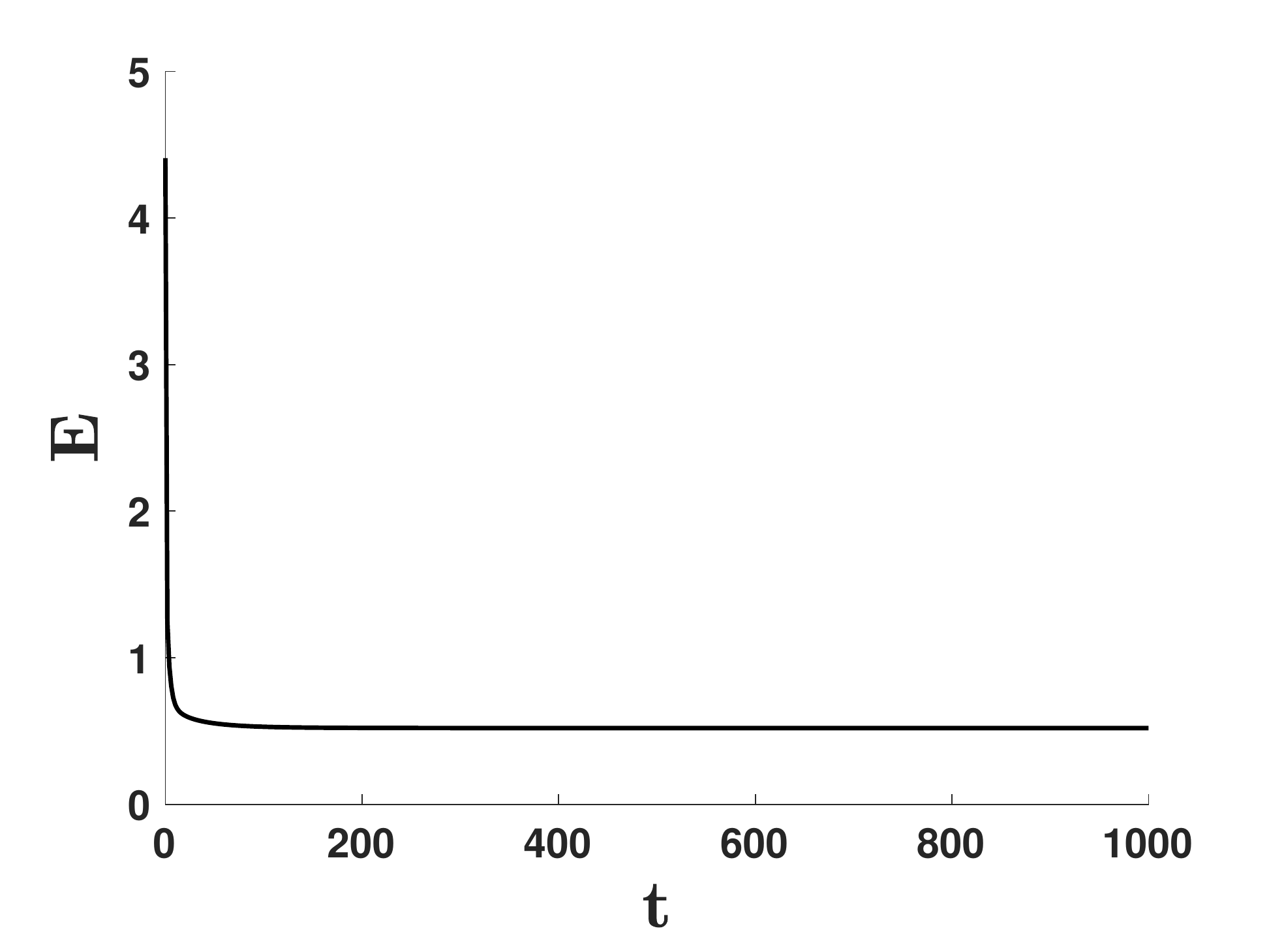}  \\
 (a) & (b) 
 \end{tabular}
 \begin{center}
 \end{center}
\caption{Asymptotic flocking of $N=10$ particles along a closed geodesic path. (a) The black dots and the red diamonds indicate  the particle locations at $t=0$ and $t=2000$, respectively. The particles have aligned their velocities (in the parallel transport sense) along the closed geodesic path shown by small blue dots, and will move along this path indefinitely. (b) Energy decay over time. Note that the energy approaches a non-zero value as $t \to \infty$.}
\label{fig:flocking}
\end{center}
\end{figure}

\section{Conclusion} \label{sec:6}
\setcounter{equation}{0}
In this paper, we have studied the Cucker-Smale model on $\mathrm{SO}(3)$ which exhibits  collective behaviors of a rotation matrix ensemble. In earlier work \cite{H-K-S}, a general and abstract Cucker-Smale model was proposed on a connected, complete and smooth Riemannian manifold using the covariant derivative and parallel transport of tangent vectors along length minimizing geodesics. 
Thus, for the explicit dynamics and numerical simulations of particles lying on a given specific Riemannian manifold, one needs to find explicit forms for the covariant derivative and parallel transport in terms of state variables and the underlying geometric information. Fortunately, for the special orthogonal group, we can explicitly calculate the covariant derivative and parallel transport which results in the explicit representation of the Cucker-Smale model on $\mathrm{SO}(3)$. For the proposed model, we used a Lyapunov approach based on an energy functional. By deriving a dissipative estimate, we use La Salle's invariance principle to conclude the velocity alignment without any a priori conditions. 

Of course, there are still various issues that have not been addressed in the current work, for example, we do not have explicit information on the spatial structure of emerging asymptotic configurations and asymptotic velocity, due to the lack of enough conservation laws, etc. These interesting issues will be addressed in future work. 

\begin{appendix}
\section{Proofs of several lemmas} \label{App-A}
\setcounter{equation}{0}
In this appendix, we present proofs of several lemmas and details to various calculations employed in the main body of the paper. 

\subsection{Proof of  Lemma \ref{L2.2}.} \label{appendix:L2.2} By direct calculation, one has 
\begin{equation}
\label{eqn:xhat2}
\widehat{\mathbf{x}}^2=\begin{bmatrix}
-x_2^2-x_3^2&x_1x_2&x_1x_3\\
x_1x_2&-x_1^2-x_3^2&x_2x_3\\
x_1x_3&x_2x_3&-x_1^2-x_2^2
\end{bmatrix}.
\end{equation}
This yields
\begin{align*}
\widehat{\mathbf{x}}^3&=\begin{bmatrix}
-x_2^2-x_3^2&x_1x_2&x_1x_3\\
x_1x_2&-x_1^2-x_3^2&x_2x_3\\
x_1x_3&x_2x_3&-x_1^2-x_2^2
\end{bmatrix}\begin{bmatrix}
0&-x_3&x_2\\
x_3&0&-x_1\\
-x_2&x_1&0
\end{bmatrix}\\[3pt]
&=
\begin{bmatrix}
0&x_3(x_1^2+x_2^2+x_3^2)&-x_2(x_1^2+x_2^2+x_3^2)\\
-x_3(x_1^2+x_2^2+x_3^2)&0&x_1(x_1^2+x_2^2+x_3^2)\\
x_2(x_1^2+x_2^2+x_3^2)&-x_1(x_1^2+x_2^2+x_3^2)&0
\end{bmatrix} \\[3pt]
&=-\theta ^2\widehat{\mathbf{x}},
\end{align*}
which shows the first identity.

We verify the second identity for $\alpha=1$. The other cases can be treated similarly. 
\[
E_1\widehat{\mathbf{x}}=\begin{bmatrix}
0&0&0\\
0&0&-1\\
0&1&0
\end{bmatrix}\begin{bmatrix}
0&-x_3&x_2\\
x_3&0&-x_1\\
-x_2&x_1&0
\end{bmatrix}=\begin{bmatrix}
0&0&0\\
x_2&-x_1&0\\
x_3&0&-x_1
\end{bmatrix}.
\]
Furthermore,
\begin{align*}
x_1\widehat{\mathbf{x}}^2+\widehat{\mathbf{x}}^2E_1\widehat{\mathbf{x}}&=\widehat{\mathbf{x}}^2(x_1I+E_1\widehat{\mathbf{x}})=\widehat{\mathbf{x}}^2\begin{bmatrix}
x_1&0&0\\
x_2&0&0\\
x_3&0&0
\end{bmatrix}\\[3pt]
&=\begin{bmatrix}
-x_2^2-x_3^2&x_1x_2&x_1x_3\\
x_1x_2&-x_1^2-x_3^2&x_2x_3\\
x_1x_3&x_2x_3&-x_1^2-x_2^2
\end{bmatrix}\begin{bmatrix}
x_1&0&0\\
x_2&0&0\\
x_3&0&0
\end{bmatrix}
=\begin{bmatrix}
0&0&0\\
0&0&0\\
0&0&0
\end{bmatrix}.
\end{align*}
Similarly one check the identity for $\alpha=2$ and $\alpha=3$.
\medskip


\subsection{Proof of  Lemma \ref{L2.3}.} \label{appendix:L2.3} We provide derivations for the following three identities:
\begin{align*}
\begin{aligned}
& ~\langle \widehat{\mathbf{x}}^2,\widehat{\mathbf{x}}^2\rangle_{\mathrm{F}} =2\theta ^4. \\[2pt]
& ~\langle \widehat{\mathbf{x}}^2, \{E_\alpha, \widehat{\mathbf{x}}\}\rangle_{\mathrm{F}} =4 x_\alpha\theta ^2. \\[2pt]
& ~\langle \{E_\alpha, \widehat{\mathbf{x}}\}, \{E_\beta, \widehat{\mathbf{x}}\}\rangle_{\mathrm{F}} =6x_\alpha x_\beta+2\delta_{\alpha\beta}\theta ^2. 
\end{aligned}
\end{align*}

\noindent $\bullet$~(Derivation of the first identity):~we use \eqref{eqn:xhat2} to calculate
\begin{align*}
\langle \widehat{\mathbf{x}}^2, \widehat{\mathbf{x}}^2\rangle_{\mathrm{F}}&=(-x_2^2-x_3^2)^2+(x_1x_2)^2+(x_1x_3)^2+(x_1x_2)^2+(-x_1^2-x_3^2)^2\\
&+(x_2x_3)^2+
(x_1x_3)^2+(x_2x_3)^2+(-x_1^2-x_2^2)^2\\
&=2(x_1^2+x_2^2+x_3^2)^2=2\theta ^4.
\end{align*}

\vspace{0.1cm}

\noindent $\bullet$~(Derivation of the second identity):~we use the commutativity of the trace to get
\begin{equation} \label{Ap-1}
\langle \widehat{\mathbf{x}}^2, \{E_\alpha, \widehat{\mathbf{x}}\}\rangle_{\mathrm{F}}=\mathrm{tr}(\widehat{\mathbf{x}}^2(E_\alpha\widehat{\mathbf{x}}+\widehat{\mathbf{x}}E_\alpha))=2\mathrm{tr}(\widehat{\mathbf{x}}^3E_\alpha).
\end{equation}
Then use the first identity in Lemma \ref{L2.1} and the fact that $\hbx$ is skew-symmetric to find
\begin{equation} \label{Ap-2}
2\mathrm{tr}(\widehat{\mathbf{x}}^3E_\alpha)=-2\theta ^2\mathrm{tr}(\widehat{\mathbf{x}}E_\alpha)=2\theta ^2\langle \widehat{\mathbf{x}}, E_\alpha\rangle_{\mathrm{F}}.
\end{equation}
Finally, we combine \eqref{Ap-1}, \eqref{Ap-2} and  the third identity in Lemma \ref{L2.2} to obtain the desired second identity.

\vspace{0.1cm}

\noindent $\bullet$~(Derivation of the third identity):~Note that 
\begin{align}
\begin{aligned} \label{Ap-3}
&\langle \{E_\alpha, \widehat{\mathbf{x}}\}, \{E_\beta, \widehat{\mathbf{x}}\}\rangle_{\mathrm{F}} \\
& \hspace{0.5cm} =x_\kappa x_\lambda
\langle \{E_\alpha, E_\kappa\}, \{E_\beta, E_\lambda\}\rangle_{\mathrm{F}}  \\
& \hspace{0.5cm} =x_\kappa x_\lambda\langle E_\alpha E_\kappa+E_\kappa E_\alpha, E_\beta E_\lambda+E_\lambda E_\beta\rangle_{\mathrm{F}} \\
& \hspace{0.5cm} =x_\kappa x_\lambda\mathrm{tr}(E_\alpha E_\kappa E_\beta E_\lambda+E_\kappa E_\alpha E_\beta E_\lambda+E_\alpha E_\kappa E_\lambda E_\beta+E_\kappa E_\alpha E_\lambda E_\beta).
\end{aligned}
\end{align}
Next, we want to express
\[
\mathrm{tr}(E_\iota E_\kappa E_\mu E_\nu)
\]
for general $\iota, \kappa, \mu, \nu \in\{1,2, 3\}$. To simplify this term, we will use the notation introduced in Remark \ref{R2.1}. We use Einstein's convention (repeated index means summation) to find
\begin{align}
\begin{aligned} \label{Ap-4}
\mathrm{tr}(E_\iota E_\kappa E_\mu E_\nu)&=[E_\iota E_\kappa E_\mu E_\nu]_{\alpha\alpha}=[E_\iota]_{\alpha\beta}[E_\kappa]_{\beta\gamma}[E_\mu]_{\gamma\delta}[E_\nu]_{\delta\alpha}\\
&=\epsilon_{\iota\alpha\beta}\epsilon_{\kappa\beta\gamma}\epsilon_{\mu\gamma\delta}\epsilon_{\nu\delta\alpha}=(\epsilon_{\beta \iota\alpha}\epsilon_{\beta\gamma \kappa})(\epsilon_{\delta \mu\gamma}\epsilon_{\delta\alpha \nu})\\
&=(\delta_{\iota\gamma}\delta_{\alpha \kappa}-\delta_{\iota\kappa}\delta_{\alpha\gamma})(\delta_{\mu\alpha}\delta_{\gamma \nu}-\delta_{\mu \nu}\delta_{\gamma\alpha})\\
&=\delta_{\iota\gamma}\delta_{\alpha \kappa}\delta_{\mu\alpha}\delta_{\gamma \nu}-\delta_{\iota\gamma}\delta_{\alpha \kappa}\delta_{\mu \nu}\delta_{\gamma\alpha}-\delta_{\iota\kappa}\delta_{\alpha\gamma}\delta_{\mu\alpha}\delta_{\gamma \nu}+\delta_{\iota\kappa}\delta_{\alpha\gamma}\delta_{\mu \nu}\delta_{\gamma\alpha}\\
&=\delta_{\iota \nu}\delta_{\kappa \mu}-\delta_{\iota\kappa}\delta_{\mu \nu}-\delta_{\iota\kappa}\delta_{\mu \nu}+\delta_{\iota\kappa}\delta_{\mu \nu}\delta_{\alpha\alpha}\\
&=\delta_{\iota \nu}\delta_{\kappa \mu}+\delta_{\iota\kappa}\delta_{\mu \nu},
\end{aligned}
\end{align}
where we used $\delta_{\alpha\alpha}=3$ in last equality.  Now we combine \eqref{Ap-3} and \eqref{Ap-4} to get
\begin{align*}
&\langle \{E_\alpha, \widehat{\mathbf{x}}\}, \{E_\beta, \widehat{\mathbf{x}}\}\rangle_{\mathrm{F}} \\
&\qquad =x_\kappa x_\lambda\big( \delta_{\alpha \kappa}\delta_{\beta \lambda}+\delta_{\kappa\beta}\delta_{\alpha \lambda}+\delta_{\kappa\alpha}\delta_{\beta \lambda}+\delta_{\alpha \beta}\delta_{\lambda\kappa}+\delta_{\alpha \kappa}\delta_{\lambda\beta}+\delta_{\kappa \lambda}\delta_{\alpha \beta} +\delta_{\kappa\alpha}\delta_{\lambda\beta}+\delta_{\alpha \lambda}\delta_{\kappa\beta}\big)\\
&\qquad =x_\kappa x_\lambda\big(4\delta_{\alpha \kappa}\delta_{\beta \lambda}+2\delta_{\beta \kappa}\delta_{\alpha \lambda}+2\delta_{\alpha \beta}\delta_{\kappa \lambda}\big)\\
&\qquad =6x_\alpha x_\beta+2\delta_{\alpha \beta}\theta ^2.
\end{align*}

\medskip

\subsection{Inverse of the metric tensor} \label{appendix:metric-inv} Let  $g_{\alpha\beta}$ be the metric tensor given in Proposition \ref{P2.2}(i). We look for the coefficients $g^{\beta \gamma}$ of the metric's inverse in the following ansatz:
\begin{equation}
\label{eqn:g-inv-decomp}
g^{\beta\gamma}=\mathcal{A}x_\beta x_\gamma+\mathcal{B}\delta^{\beta\gamma}.
\end{equation}
Now we calculate $g_{\alpha\beta}g^{\beta\gamma}$ using \eqref{eqn:metric} as follows:
\begin{align*}
g_{\alpha\beta}g^{\beta\gamma}&=\left(\left(\frac{2\cos\theta -2+\theta ^2}{\theta ^4}\right) x_\alpha x_\beta +\frac{2(1-\cos\theta )}{\theta ^2}\delta_{\alpha\beta} \right)(\mathcal{A}x_\beta x_\gamma+\mathcal{B}\delta^{\beta\gamma})\\
&=\mathcal{A}\left(\frac{2\cos\theta -2+\theta ^2}{\theta ^2}\right) x_\alpha x_\gamma+\mathcal{A}\left(\frac{2(1-\cos\theta )}{\theta ^2}\right) x_\alpha x_\gamma\\
&+\mathcal{B} \left(\frac{2\cos\theta -2+\theta ^2}{\theta ^4}\right) x_\alpha x_\gamma +\mathcal{B}\left(\frac{2(1-\cos\theta )}{\theta ^2}\right) \delta_\alpha^\gamma,
\end{align*}
where we used $x_\beta x_\beta = \theta^2$ in the second equality. \newline

Since we require to have $g_{\alpha\beta}g^{\beta\gamma}=\delta_{\alpha}^\gamma$, we simply set
\[ \mathcal{B}\left(\frac{2(1-\cos\theta )}{\theta ^2}\right)=1, \quad \mathcal{A}\left(\frac{2\cos\theta -2+\theta ^2}{\theta ^2}\right)+\mathcal{A}\left(\frac{2(1-\cos\theta )}{\theta ^2}\right)+\mathcal{B}\left(\frac{2\cos\theta -2+\theta ^2}{\theta ^4}\right)=0.
\]
Then, this yields
\begin{align}
\label{eqn:AB}
\begin{aligned}
\mathcal{A}=\frac{2\cos\theta -2+\theta ^2}{2\theta ^2(\cos\theta -1)},\quad
\mathcal{B}=\frac{\theta ^2}{2(1-\cos\theta )}.
\end{aligned}
\end{align}
Finally, \eqref{eqn:g-inv-decomp} and \eqref{eqn:AB} imply the desired relation for $g^{\alpha \beta}$ in Proposition \ref{P2.2}(ii). 

\medskip

\subsection{Christoffel symbols \eqref{eqn:Christoffel}}\label{appendix:Christoffel} Note that 
\[
g_{\alpha \beta}= \left(\frac{2\cos\theta -2+\theta ^2}{\theta ^4}\right) x_\alpha x_\beta +\frac{2(1-\cos\theta )}{\theta ^2}\delta_{\alpha\beta}.
\]
We differentiate  \eqref{eqn:metric} to find
\begin{align*}
\partial_\gamma g_{\alpha \beta}(\mathbf{x})&=\left(\frac{2\cos\theta -2+\theta ^2}{\theta ^4}\right) (x_\alpha\delta_{\beta\gamma}+x_\beta\delta_{\alpha\gamma})\\
&\quad +\left(\left(\frac{8-8\cos\theta -2\theta \sin\theta -2\theta ^2}{\theta ^5}\right) x_\alpha x_\beta+2 \left(\frac{-2+2\cos\theta +\theta \sin\theta }{\theta ^3}\right) \delta_{\alpha\beta} \right)\frac{x_\gamma}{\theta }\\
&= \left(\frac{2\cos\theta -2+\theta ^2}{\theta ^4}\right)(x_\alpha\delta_{\beta\gamma}+x_\beta\delta_{\alpha\gamma}+2x_\gamma\delta_{\alpha\beta})+2 \left(\frac{\theta \sin\theta -\theta ^2}{\theta ^4}\right) \delta_{\alpha\beta}x_\gamma\\
&\quad +\left(\frac{8-8\cos\theta -2\theta \sin\theta -2\theta ^2}{\theta ^6}\right)x_\alpha x_\beta x_\gamma.
\end{align*}
This yields
\begin{align}
\label{eqn:deriv-sum}
\begin{aligned}
&\partial_\beta g_{\gamma\alpha}+\partial_\alpha g_{\gamma\beta}-\partial_\gamma g_{\alpha\beta}  =2 \left(\frac{2\cos\theta -2+\theta ^2}{\theta ^4}\right) (x_\beta\delta_{\alpha\gamma}+x_\alpha\delta_{\beta\gamma})\\
&+2\left(\frac{\theta \sin\theta -\theta ^2}{\theta ^4}\right) (\delta_{\alpha\gamma}x_\beta+\delta_{\beta\gamma}x_\alpha-\delta_{\alpha\beta}x_\gamma)
+\left(\frac{8-8\cos\theta -2\theta \sin\theta -2\theta ^2}{\theta ^6}\right) x_\alpha x_\beta x_\gamma.
\end{aligned}
\end{align}
Now, we use  explicit relation of the inverse metric together with $x_\kappa x_\kappa = \theta^2$ to see
\begin{align*}
g^{\lambda\gamma} x_\gamma&=\left(\frac{2\cos\theta -2+\theta ^2}{2\theta ^2(\cos\theta -1)}\right) x_\lambda \theta ^2+\left(\frac{\theta ^2}{2(1-\cos \theta)}\right) x_\lambda=x_\lambda.
\end{align*}
Now use \eqref{eqn:deriv-sum} to find the Christoffel symbols:
\begin{align}
\begin{aligned} \label{Z-1}
\Gamma^\lambda_{\alpha\beta}& =\frac{1}{2}g^{\lambda \gamma}(\partial_\beta g_{\gamma \alpha}+\partial_\alpha g_{\gamma \beta}-\partial_\gamma g_{\alpha\beta}) \\
&=g^{\lambda\gamma}(x_\beta\delta_{\alpha\gamma}+x_\alpha\delta_{\beta\gamma})\left(\frac{2\cos\theta -2+\theta ^2}{\theta ^4}\right)+g^{\lambda\gamma}(\delta_{\alpha\gamma}x_\beta+\delta_{\beta\gamma}x_\alpha-\delta_{\alpha\beta}x_\gamma)\left(\frac{\theta \sin\theta -\theta ^2}{\theta ^4}\right) \\[2pt]
& \quad +g^{\lambda\gamma}x_\alpha x_\beta x_\gamma\left(\frac{4-4\cos\theta -\theta \sin\theta -\theta ^2}{\theta ^6}\right)\\
&= \left( g^{\lambda \alpha} x_\beta + g^{\lambda \beta} x_\alpha  \right) \left(\frac{2\cos\theta -2+\theta ^2}{\theta ^4}\right)+(g^{\lambda \alpha} x_\beta + g^{\lambda \beta} x_\alpha -\delta_{\alpha\beta}x_\lambda)\left(\frac{\theta \sin\theta -\theta ^2}{\theta ^4}\right)\\[2pt]
&= \left( g^{\lambda \alpha} x_\beta + g^{\lambda \beta} x_\alpha  \right) \left(\frac{2\cos\theta -2+\theta \sin \theta}{\theta ^4}\right) -\delta_{\alpha\beta}x_\lambda\left(\frac{\theta \sin\theta -\theta ^2}{\theta ^4}\right) \\
&\quad +x_\alpha x_\beta x_\lambda\left(\frac{4-4\cos\theta -\theta \sin\theta -\theta ^2}{\theta ^6}\right).
\end{aligned}
\end{align}
By \eqref{eqn:g-inv-decomp} and \eqref{eqn:AB}, write:
\begin{align}
\begin{aligned} \label{Z-2}
g^{\lambda \alpha} x_\beta + g^{\lambda \beta} x_\alpha &=2\mathcal{A}x_\lambda x_\alpha x_\beta+\mathcal{B}(x_\alpha\delta_\beta^\lambda+x_\beta\delta_\alpha^\lambda) \\[3pt]
&=\frac{2\cos\theta -2+\theta ^2}{\theta ^2(\cos\theta -1)}x_\lambda x_\alpha x_\beta+\frac{\theta ^2}{2(1-\cos\theta )}(x_\alpha\delta_\beta^\lambda+x_\beta\delta_\alpha^\lambda).
\end{aligned}
\end{align}
Finally, we combine \eqref{Z-1} and \eqref{Z-2} to derive \eqref{eqn:Christoffel} for the Christoffel coefficients.

\medskip

\subsection{Equation  \eqref{eqn:geod-check}  for geodesics} \label{appendix:geod} Let $x_\alpha(t) = t u_\alpha.$ Then, \eqref{eqn:geod-check} is equivalent to show
\[
\Gamma^\gamma_{\alpha \beta}(\bx(t)) x_\alpha(t) x_\beta(t) = 0.
\]
We multiply $\Gamma^\gamma_{\alpha \beta} (\bx)$ from \eqref{eqn:Christoffel} by $x_\alpha x_\beta$,  and then sum up the resulting relation over $\alpha$ and $\beta$, and use $x_\alpha x_\alpha = \| \bx\|^2 = \theta^2$ to obtain
\begin{align*}
\Gamma^\gamma_{\alpha \beta} (\bx) x_\alpha x_\beta& = \frac{(\sin \theta + \theta)(\cos \theta -1) + \theta^2 \sin \theta}{\theta^5 (\cos \theta -1)}  x_\alpha x_\beta x_\gamma x_\alpha x_\beta \\[3pt]
&\quad + \frac{2\cos\theta -2+\theta \sin\theta}{2 \theta^2 (1-\cos\theta )}(x_\alpha\delta_\beta^\gamma+x_\beta\delta_\alpha^\gamma) x_\alpha x_\beta - \frac{\sin\theta -\theta}{\theta^3} \delta_{\alpha\beta}x_\gamma x_\alpha x_\beta \\[3pt]
&= \frac{(\sin \theta + \theta)(\cos \theta -1) + \theta^2 \sin \theta}{\theta(\cos \theta -1)}  x_\gamma + \frac{2\cos\theta -2+\theta \sin\theta}{1-\cos\theta} x_\gamma - \frac{\sin\theta -\theta}{\theta} x_\gamma.
\end{align*}
Note that $x_\gamma$ factors out in the R.H.S. above. By a direct calculation, one can see that sum of coefficients equal to zero:
\[
 \frac{(\sin \theta + \theta)(\cos \theta -1) + \theta^2 \sin \theta}{\theta(\cos \theta -1)} + \frac{2\cos\theta -2+\theta \sin\theta}{1-\cos\theta} - \frac{\sin\theta -\theta}{\theta} =0.
\]
This establishes the desired result.

\medskip

\subsection{Closed form expression \eqref{B-11} of parallel transport} \label{appendix:pt-closed} 
We use \eqref{eqn:XalphaR1}, \eqref{eqn:v1-mod} and $v_\alpha(1) u_\alpha = \calC$ (see \eqref{eqn:calC}) to express $V(1)$ as follows:
\begin{align*}
V(1)&=v_\alpha(1)X_\alpha(R_1)\\
&=\left(\frac{-\sin\theta+\theta\cos\theta}{\theta^3}\right)\mathcal{C}R_0\widehat{\mathbf{u}}+\frac{\sin\theta}{\theta} \left(\frac{\theta v_\alpha(0)}{2\sin\frac{\theta }{2}}+\mathcal{C}u_\alpha\left(\frac{1}{\theta^2}-\frac{1}{2\theta\sin\frac{\theta}{2}}\right)\right)R_0E_\alpha\\
&\quad +\left(\frac{\theta\sin\theta-2(1-\cos\theta)}{\theta^4}\right)\mathcal{C}R_0\widehat{\mathbf{u}}^2 \\
&\quad +\frac{1-\cos\theta}{\theta^2}\left(\frac{\theta v_\alpha(0)}{2\sin\frac{\theta }{2}}+\mathcal{C}u_\alpha\left(\frac{1}{\theta^2}-\frac{1}{2\theta\sin\frac{\theta}{2}}\right)\right)(R_0E_\alpha\widehat{\mathbf{u}}+R_0\widehat{\mathbf{u}}E_\alpha).
\end{align*}
By \eqref{eqn:V0} and \eqref{eqn:XalphaR1}, one can write 
\[ V(0)=v_\alpha(0)R_0E_\alpha \quad \mbox{and} \quad \hbu = u_\alpha E_\alpha. \]
So one can continue the calculation above to get:
\begin{align*}
V(1)&=\left(\frac{-\sin\theta+\theta\cos\theta}{\theta^3}\right)\mathcal{C}R_0\widehat{\mathbf{u}}+\frac{\sin\theta}{\theta} \left(\frac{\theta V(0)}{2\sin\frac{\theta }{2}}+\mathcal{C}R_0\widehat{\mathbf{u}}\left(\frac{1}{\theta^2}-\frac{1}{2\theta\sin\frac{\theta}{2}}\right)\right)\\
&\quad +\left(\frac{\theta\sin\theta-2(1-\cos\theta)}{\theta^4}\right)\mathcal{C}R_0\widehat{\mathbf{u}}^2\\
&\quad +\frac{1-\cos\theta}{\theta^2}\left(\frac{\theta }{2\sin\frac{\theta }{2}}(V(0)\widehat{\mathbf{u}}+R_0\widehat{\mathbf{u}}R_0^\top V(0))+2R_0\widehat{\mathbf{u}}^2\mathcal{C}\left(\frac{1}{\theta^2}-\frac{1}{2\theta\sin\frac{\theta}{2}}\right)\right)\\
&=\mathcal{C}R_0\widehat{\mathbf{u}}\left(\frac{-\sin\theta+\theta\cos\theta}{\theta^3}+\frac{\sin\theta}{\theta}\left(\frac{1}{\theta^2}-\frac{1}{2\theta\sin\frac{\theta}{2}}\right) \right)\\
&\quad +\mathcal{C}R_0\widehat{\mathbf{u}}^2\left(\frac{\theta\sin\theta-2(1-\cos\theta)}{\theta^4}+2\frac{1-\cos\theta}{\theta^2}\left(\frac{1}{\theta^2}-\frac{1}{2\theta\sin\frac{\theta}{2}}\right)\right)\\
&\quad +\frac{1-\cos\theta}{\theta^2}\left(\frac{\theta }{2\sin\frac{\theta }{2}}(V(0)\widehat{\mathbf{u}}+R_0\widehat{\mathbf{u}}R_0^\top V(0))\right)+\frac{\sin\theta}{\theta} \left(\frac{\theta V(0)}{2\sin\frac{\theta }{2}}\right)\\
&=\mathcal{C}R_0\widehat{\mathbf{u}}\left(\frac{\cos\theta-\cos\frac{\theta}{2}}{\theta^2} \right)+\mathcal{C}R_0\widehat{\mathbf{u}}^2\left(\frac{\sin\theta}{\theta^3}-\frac{2\sin\frac{\theta}{2}}{\theta^3}\right)\\
&\quad +\frac{\sin\frac{\theta}{2}}{\theta}\left(V(0)\widehat{\mathbf{u}}+R_0\widehat{\mathbf{u}}R_0^\top V(0)\right)+\cos\frac{\theta }{2}V(0).
\end{align*}


\medskip

\subsection{Expression \eqref{B-12} for $A_1$} \label{appendix:A1} We use \eqref{eqn:RT},  \eqref{eqn:RT-cont} in \eqref{eqn:A1-int} and arrange terms to get
\begin{align}
\begin{aligned} \label{NN-1}
A_1&=\mathcal{C}\left(\frac{\cos\theta-\cos\frac{\theta}{2}}{\theta^2} \right)\left(\cos\theta \widehat{\mathbf{u}}-\frac{\sin\theta}{\theta}\widehat{\mathbf{u}}^2\right)
+\mathcal{C}\left(\frac{\sin\theta-2\sin\frac{\theta}{2}}{\theta^3}\right)\left(\cos\theta \widehat{\mathbf{u}}^2+\theta\sin\theta\widehat{\mathbf{u}}\right)\\
&\quad +\frac{\sin\frac{\theta}{2}}{\theta}\left(I-\frac{\sin\theta}{\theta}\widehat{\mathbf{u}}+\frac{1-\cos\theta}{\theta^2}\widehat{\mathbf{u}}^2\right)\left(A_0\widehat{\mathbf{u}}+\widehat{\mathbf{u}}A_0\right)+\cos\frac{\theta }{2}\left(I-\frac{\sin\theta}{\theta}\widehat{\mathbf{u}}+\frac{1-\cos\theta}{\theta^2}\widehat{\mathbf{u}}^2\right)A_0\\[3pt]
&=\mathcal{C}\left(\frac{1-\cos\theta\cos\frac{\theta}{2}-2\sin\theta\sin\frac{\theta}{2}}{\theta^2}\right)\widehat{\mathbf{u}}
+\mathcal{C}\left(\frac{\cos\frac{\theta}{2}\sin\theta-2\sin\frac{\theta}{2}\cos\theta}{\theta^3}\right)\widehat{\mathbf{u}}^2\\
&\quad +\frac{\sin\frac{\theta}{2}}{\theta}\left(A_0\widehat{\mathbf{u}}+\widehat{\mathbf{u}}A_0\right)
-\frac{\sin\frac{\theta}{2}\sin\theta}{\theta^2}(\widehat{\mathbf{u}}A_0\widehat{\mathbf{u}}+\widehat{\mathbf{u}}^2A_0)
+\left(\frac{\sin\frac{\theta}{2}(1-\cos\theta)}{\theta^3}\right)\left(\widehat{\mathbf{u}}^2A_0\widehat{\mathbf{u}}+\widehat{\mathbf{u}}^3A_0\right)\\
&\quad +\cos\frac{\theta }{2}A_0-\frac{\cos\frac{\theta }{2}\sin\theta}{\theta}\widehat{\mathbf{u}}A_0+\frac{\cos\frac{\theta }{2}(1-\cos\theta)}{\theta^2}\widehat{\mathbf{u}}^2A_0.
\end{aligned}
\end{align}
After some simple trigonometric manipulations, we can write \eqref{NN-1} as 
\begin{align}
\label{eqn:A1-cont}
\begin{aligned}
A_1 &=\mathcal{C}\left(\frac{1-\cos\theta\cos\frac{\theta}{2}-2\sin\theta\sin\frac{\theta}{2}}{\theta^2}\right)\widehat{\mathbf{u}}
+\mathcal{C}\left(\frac{2\sin^3\frac{\theta}{2}}{\theta^3}\right)\widehat{\mathbf{u}}^2 \\
&\quad +\frac{\sin\frac{\theta}{2}}{\theta}\left(A_0\widehat{\mathbf{u}}+\widehat{\mathbf{u}}A_0\right)
-\frac{2\sin^2\frac{\theta}{2}\cos\frac{\theta}{2}}{\theta^2}(\widehat{\mathbf{u}}A_0\widehat{\mathbf{u}}+\widehat{\mathbf{u}}^2A_0)
+\left(\frac{2\sin^3\frac{\theta}{2}}{\theta^3}\right)\left(\widehat{\mathbf{u}}^2A_0\widehat{\mathbf{u}}-\theta^2\widehat{\mathbf{u}}A_0\right)\\
&\quad +\cos\frac{\theta }{2}A_0-\frac{2\cos^2\frac{\theta }{2}\sin\frac{\theta}{2}}{\theta}\widehat{\mathbf{u}}A_0+\frac{2\sin^2\frac{\theta}{2}\cos\frac{\theta}{2}}{\theta^2}\widehat{\mathbf{u}}^2A_0\\[5pt]
&=\mathcal{C}\left(\frac{1-\cos\theta\cos\frac{\theta}{2}-2\sin\theta\sin\frac{\theta}{2}}{\theta^2}\right)\widehat{\mathbf{u}}
+\mathcal{C}\left(\frac{2\sin^3\frac{\theta}{2}}{\theta^3}\right)\widehat{\mathbf{u}}^2\\
&\quad +\frac{\sin\frac{\theta}{2}}{\theta}\left(A_0\widehat{\mathbf{u}}-\widehat{\mathbf{u}}A_0\right)
-\frac{2\sin^2\frac{\theta}{2}\cos\frac{\theta}{2}}{\theta^2}\widehat{\mathbf{u}}A_0\widehat{\mathbf{u}}
+\left(\frac{2\sin^3\frac{\theta}{2}}{\theta^3}\right)\widehat{\mathbf{u}}^2A_0\widehat{\mathbf{u}}
+\cos\frac{\theta }{2}A_0.
\end{aligned}
\end{align}
To derive further simplification of \eqref{eqn:A1-cont}, we present an elementary lemma as follows.
\begin{lemma}\label{LA.1}
The following identity holds:
\[
\mathcal{C}\widehat{\mathbf{u}}^2+\widehat{\mathbf{u}}^2R_0^\top V(0)\widehat{\mathbf{u}}=0,
\]
where the constant $\calC$ is given in \eqref{eqn:calC}.
\end{lemma}
\begin{proof}
By the second identity in Lemma \ref{L2.1}, we have
\[
u_\alpha\widehat{\mathbf{u}}^2+\widehat{\mathbf{u}}^2E_\alpha\widehat{\mathbf{u}}=0.
\]
We multiply by $v_\alpha(0)$ and sum up  the resulting relation over $\alpha$ to get
\[
v_\alpha(0)u_\alpha\widehat{\mathbf{u}}^2+\widehat{\mathbf{u}}^2v_\alpha(0)E_\alpha\widehat{\mathbf{u}}=0.
\]
By $v_\alpha(0)u_\alpha=\mathcal{C}$ and $V(0)=v_\alpha(0) R_1E_\alpha$ (see \eqref{eqn:V0} and \eqref{eqn:XalphaR1}), one can obtain the desired identity. 
\end{proof}
Now, we use Lemma \ref{LA.1} to get 
\[
\mathcal{C}\widehat{\mathbf{u}}^2=-\widehat{\mathbf{u}}^2R_0^\top V(0)\widehat{\mathbf{u}}=-\widehat{\mathbf{u}}^2A_0\widehat{\mathbf{u}}.
\]
Finally, we use this identity in the second and the fifth terms in the R.H.S. of \eqref{eqn:A1-cont} to cancel, and we derive \eqref{B-12}.

\medskip

\subsection{Proof of Lemma \ref{lem:various}} \label{proofL:various}
Take  $\bx=(x_1,x_2,x_3)$, $\by=(y_1,y_2,y_3)$ in $\bbr^3$.

By direct calculation, we have
\begin{align}\label{mpd}
\widehat{\bx}\widehat{\by}=\begin{bmatrix}
0&-x_3&x_2\\
x_3&0&-x_1\\
-x_2&x_1&0
\end{bmatrix}\begin{bmatrix}
0&-y_3&y_2\\
y_3&0&-y_1\\
-y_2&y_1&0
\end{bmatrix}=\begin{bmatrix}
-x_3y_3-x_2y_2&x_2y_1&x_3y_1\\
x_1y_2&-x_3y_3-x_1y_1&x_3y_2\\
x_1y_3&x_2y_3&-x_2y_2-x_1y_1
\end{bmatrix}.
\end{align}
This implies 
\[
\Tr(\hbx \hby) =-2(x_1y_1+x_2y_2+x_3y_3)=-2\,\bx\cdot\by,
\]
which is the first identity in \eqref{eqn:id1}.

Then, from \eqref{mpd} we have
\begin{align*}
&\hbx\hby-\hby\hbx\\
&=\begin{bmatrix}
-x_3y_3-x_2y_2&x_2y_1&x_3y_1\\
x_1y_2&-x_3y_3-x_1y_1&x_3y_2\\
x_1y_3&x_2y_3&-x_2y_2-x_1y_1
\end{bmatrix}
-\begin{bmatrix}
-x_3y_3-x_2y_2&x_1y_2&x_1y_3\\
x_2y_1&-x_3y_3-x_1y_1&x_2y_3\\
x_3y_1&x_3y_2&-x_2y_2-x_1y_1
\end{bmatrix}\\[5pt]
&=\begin{bmatrix}
0&x_2y_1-x_1y_2&x_3y_1-x_1y_3\\
x_1y_2-x_2y_1&0&x_3y_2-x_2y_3\\
x_1y_3-x_3y_1&x_2y_3-x_3y_3&0
\end{bmatrix}.
\end{align*}
Since the matrix on the r.h.s. above is  $\widehat{\bx \times \by}$, we get the second identity in \eqref{eqn:id1}.

Finally, also from \eqref{mpd}, we get
\begin{align*}
\hbx\hby\hbx&=\begin{bmatrix}
-x_3y_3-x_2y_2&x_2y_1&x_3y_1\\
x_1y_2&-x_3y_3-x_1y_1&x_3y_2\\
x_1y_3&x_2y_3&-x_2y_2-x_1y_1
\end{bmatrix}\begin{bmatrix}
0&-x_3&x_2\\
x_3&0&-x_1\\
-x_2&x_1&0
\end{bmatrix}\\[5pt]
&=\begin{bmatrix}
0&x_3(x_1y_1+x_2y_2+x_3y_3)&-x_2(x_1y_1+x_2y_2+x_3y_3)\\
-x_3(x_1y_1+x_2y_2+x_3y_3)&0&x_1(x_1y_1+x_2y_2+x_3y_3)\\
x_2(x_1y_1+x_2y_2+x_3y_3)&-x_1(x_1y_1+x_2y_2+x_3y_3)&0
\end{bmatrix} \\[5pt]
&=-(\bx\cdot\by)\hbx.
\end{align*}

\medskip

\subsection{Proof of Lemma \ref{L4.1}.}\label{proofL4.1}
(i) By the explicit formula of $R_i(t)$ and $R_k(t)$, we have
\[
R_i(t)=R_i^0 \exp(t A_i),\quad R_k(t)= R_k^0 \exp(t A_k ).
\]
Recall that $A_i =\widehat{\veeA}_i = \theta \hbn_i$ and $A_k = \widehat{\veeA}_k = \theta \hbn_k$ (see \eqref{eqn:aiak}). Hence, we have:
\begin{align}
\begin{aligned}
\label{eqn:term1}
& \langle  \hbn_i , R_k(t)^\top  R_i(t) \rangle_{\mathrm{F}}  = \Tr( \hbn_i  R_i(t)^\top  R_k(t))  = \Tr ( \hbn_i  \exp(-t \widehat{\veeA}_i) {R_i^0}^\top  R_k^0 \exp(t \widehat{\veeA}_k)) \\
& \hspace{0.5cm}  = \Tr \left(\hbn_i  (I - \sin(\theta t) \hbn_i  +(1-\cos(\theta t))\hbn_i^2)  R^0 (I + \sin(\theta t) \hbn_k +(1-\cos(\theta t)) \hbn_k^2) \right),
\end{aligned}
\end{align} 
where we used Rodrigues's formula and the notation $\Rz = {R_i^0}^\top  R_k^0$. 

Note that 
\begin{align}
\begin{aligned} \label{New-3}
& \hbn_i  (I - \sin(\theta t) \hbn_i  +(1-\cos(\theta t))\hbn_i^2)  \Rz (I + \sin(\theta t) \hbn_k +(1-\cos(\theta t)) \hbn_k^2) \\
& \qquad =  (\hbn_i  - \sin(\theta t) \hbn_i^2  - (1-\cos(\theta t))\hbn_i))  \Rz (I + \sin(\theta t) \hbn_k +(1-\cos(\theta t)) \hbn_k^2), \\
& \qquad =  (\cos(\theta t) \hbn_i - \sin(\theta t) \hbn_i^2)  \Rz (I + \sin(\theta t) \hbn_k +(1-\cos(\theta t)) \hbn_k^2),
\end{aligned}
\end{align} 
where we used that $\hbn_i^3 = -\hbn_i$ for the second equal sign. Now, by the trigonometric identities:
\[
\sin^2(\theta t) = \frac{1-\cos(2 \theta t)}{2}, \quad  \text{ and } \quad \cos^2(\theta t) =  \frac{1+\cos(2 \theta t)}{2},
\]
we write the R.H.S. of \eqref{New-3} as a linear combination of trigonometric functions:
\begin{align}
\begin{aligned}
\label{eqn:calc1}
&  (\cos(\theta t) \hbn_i - \sin(\theta t) \hbn_i^2) \Rz (I + \sin(\theta t) \hbn_k +(1-\cos(\theta t)) \hbn_k^2)  \\
& \hspace{0.2cm} = - \frac{1}{2} \hbn_i \Rz \hbn_k^2- \frac{1}{2} \hbn_i^2 \Rz \hbn_k +(- \hbn_i^2 \Rz - \hbn_i^2 \Rz \hbn_k^2) \sin(\theta t) + (\hbn_i \Rz + \hbn_i \Rz \hbn_k^2) \cos(\theta t)  \\
& \hspace{0.2cm}+ (\frac{1}{2} \hbn_i \Rz \hbn_k + \frac{1}{2} \hbn_i^2 \Rz \hbn_k^2) \sin(2 \theta t) +(-\frac{1}{2} \hbn_i \Rz \hbn_k^2+ \frac{1}{2} \hbn_i^2 \Rz \hbn_k ) \cos(2 \theta t).
\end{aligned}
\end{align} 
Combining \eqref{eqn:term1}, \eqref{New-3} and \eqref{eqn:calc1} leads to identity (i) in Lemma \ref{L4.1}.
\vspace{0.2cm}

\noindent (ii) Similarly, one has
\begin{align}
\begin{aligned}
\label{eqn:term2}
& \langle  \hbn_k , R_k(t)^\top  R_i(t) \rangle_{\mathrm{F}} = \Tr( \hbn_k  R_i(t)^\top  R_k(t))  =  \Tr(R_i(t)^\top  R_k(t) \hbn_k ) \\
& \hspace{0.5cm} = \Tr ((I - \sin(\theta t) \hbn_i  +(1-\cos(\theta t))\hbn_i^2)  \Rz (I + \sin(\theta t) \hbn_k +(1-\cos(\theta t)) \hbn_k^2) \hbn_k).
\end{aligned}
\end{align} 
Note that 
\begin{align}
\begin{aligned} \label{New-4}
&(I - \sin(\theta t) \hbn_i  +(1-\cos(\theta t))\hbn_i^2)  \Rz (I + \sin(\theta t) \hbn_k +(1-\cos(\theta t)) \hbn_k^2) \hbn_k \\
& \qquad = (I - \sin(\theta t) \hbn_i  +(1-\cos(\theta t))\hbn_i^2)  \Rz (\hbn_k + \sin(\theta t) \hbn_k ^2- (1-\cos(\theta t)) \hbn_k) \\
& \qquad =  (I - \sin(\theta t) \hbn_i  +(1-\cos(\theta t))\hbn_i^2)  \Rz (\cos(\theta t)) \hbn_k + \sin(\theta t) \hbn_k ^2).
\end{aligned}
\end{align}
By writing the R.H.S. as a linear combination of trigonometric functions we find:
\begin{align}
\begin{aligned}
\label{eqn:calc2}
& (I - \sin(\theta t) \hbn_i  +(1-\cos(\theta t))\hbn_i^2)  \Rz (\cos(\theta t)) \hbn_k + \sin(\theta t) \hbn_k ^2) \\
& = -\frac{1}{2} \hbn_i \Rz \hbn_k^2 -\frac{1}{2} \hbn_i^2 \Rz \hbn_k + (\Rz \hbn_k^2 +  \hbn_i^2 \Rz \hbn_k^2) \sin(\theta t) + ( \Rz \hbn_k + \hbn_i^2 \Rz \hbn_k) \cos(\theta t) \\
&  + (-\frac{1}{2} \hbn_i \Rz \hbn_k -\frac{1}{2} \hbn_i^2 \Rz \hbn_k^2) \sin(2 \theta t) + (\frac{1}{2} \hbn_i \Rz \hbn_k^2 - \frac{1}{2} \hbn_i^2 \Rz \hbn_k) \cos(2 \theta t).
\end{aligned}
\end{align}
Combining \eqref{eqn:term2}, \eqref{New-4} and \eqref{eqn:calc2} leads to identity (ii).
\qed
\end{appendix}

\bibliographystyle{abbrv}

\bibliographystyle{amsplain}

\end{document}